\documentclass[a4paper, 11pt,reqno]{amsart}
\usepackage{amsfonts, amsthm, amssymb, amsmath, stackengine, scalerel}
\usepackage{mathrsfs,array,tikz-cd}
\usepackage{hyperref}
\usepackage{eucal,fullpage,times,color,enumerate,accents, comment}
\usepackage[all]{xy}
\usepackage{xr}
\usepackage{url}
\usepackage{enumitem}
\usepackage[new]{old-arrows}
\usepackage{extpfeil}
\usepackage{turnstile} 
\usepackage{verbatim}  
\usepackage{graphics,float} 
\usetikzlibrary{graphs,decorations.pathmorphing,decorations.markings}
\input xy
\xyoption{all}
\usepackage{adjustbox}

\usepackage{algorithm,caption}  
\DeclareCaptionFont{lightblue}{\color{teal}}
\newcommand{\EXP}[1]{{\color{teal} {\tiny \bf (Exposition:)} {\bf #1}}}

\newextarrow{\xbigtoto}{{20}{20}{20}{20}}
{\bigRelbar\bigRelbar{\bigtwoarrowsleft\rightarrow\rightarrow}}

\newcommand{\calO}{{\mathcal{O}}}

\newcommand{\calU}{{\mathcal{U}}}
\newcommand{\calV}{{\mathcal{V}}}
\newcommand{\calA}{{\mathcal{A}}}

\newcommand{\calD}{{\mathcal{D}}}

\newcommand{\calR}{{\mathcal{R}}}

\newcommand{\calI}{\mathcal{I}}

\newcommand{\calH}{\mathcal{H}}

\newcommand{\calP}{\mathcal{P}}

\newcommand{\id}{\mathrm{id}}

\newcommand{\m}{\mathfrak{m}}

\newcommand{\N}{\mathbb{N}}

\DeclareFontFamily{U}{wncy}{}
\DeclareFontShape{U}{wncy}{m}{n}{<->wncyr10}{}
\DeclareSymbolFont{mcy}{U}{wncy}{m}{n}
\DeclareMathSymbol{\Sh}{\mathord}{mcy}{"58}

\DeclareSymbolFont{matha}{OML}{txmi}{m}{it}
\DeclareMathSymbol{\varv}{\mathord}{matha}{118}

\newcommand{\dom}{\mathrm{dom}}
\newcommand{\Eff}{\mathrm{Eff}}

\newcommand{\rstr}{\!\!\upharpoonright\!\!_}

\newcommand{\lranglet}[2]{\langle#1,#2\rangle}

\def\forkindep{\mathrel{\raise0.2ex\hbox{\ooalign{\hidewidth$\vert$\hidewidth\cr\raise-0.9ex\hbox{$\smile$}}}}}

\newcommand{\Fin}{\mathrm{Fin}}

\newcommand{\code}[1]{\lceil{#1}\rceil}

\newcommand{\Pw}{\calP(\w)}

\newcommand{\dn}{{\lnot\lnot}}

\newcommand{\ww}{{^{\omega}{\omega}} }
\newcommand{\lww}{{^{<\omega}{\omega}} }

\newcommand{\lk}{\leq_{\mathrm{K}}}

\newcommand{\glt}{\leq_{\mathrm{LT}}^{\mathsf{o}}}
\newcommand{\elt}{\equiv_{\mathrm{LT}}^{\mathsf{o}}}

\newcommand{\w}{\omega}

\newcommand{\clt}{\leq_{\mathrm{LT}}}

\newcommand{\sclt}{<_{\mathrm{LT}}}
\newcommand{\sglt}{<^{\circ}_{\mathrm{LT}}}
\newcommand{\Sumf}{\mathrm{Sum}_f}
\newcommand{\Sumn}{{\mathrm{Sum}_{1/n}}}
\newcommand{\Sump}{\mathrm{Sum}_P}
\newcommand{\Sumq}{\mathrm{Sum}_Q}
\newcommand{\Sums}{\mathrm{Sum}_S}

\newcommand{\pcolon}{\colon\!\!\subseteq}

\newcommand{\tto}{\rightrightarrows}

\newcommand{\fr}{\mbox{}^\smallfrown}

\newcommand{\dar}{\!\!\downarrow}
\newcommand{\uar}{\!\!\uparrow}

\newcommand{\Bmid}{\,\mathrel{\Big|}\,} 
\newcommand{\bmid}{\mathrel{\big|}} 

\newcommand{\Denz}{\mathrm{Den}_0}

\newcommand{\FS}{\mathrm{FS}}

\newcommand{\Ram}{\calR\mathrm{am}}
\newcommand{\n}{\mathfrak{n}}
\newcommand{\Tit}{\big[T[i]\big]^2}

\newcommand{\cx}{\code{x}}

\makeatletter
\newcommand*{\relrelbarsep}{.386ex}
\newcommand*{\relrelbar}{%
	\mathrel{%
		\mathpalette\@relrelbar\relrelbarsep
	}%
}
\newcommand*{\@relrelbar}[2]{%
	\raise#2\hbox to 0pt{$\m@th#1\relbar$\hss}%
	\lower#2\hbox{$\m@th#1\relbar$}%
}
\providecommand*{\rightrightarrowsfill@}{%
	\arrowfill@\relrelbar\relrelbar\rightrightarrows
}
\providecommand*{\leftleftarrowsfill@}{%
	\arrowfill@\leftleftarrows\relrelbar\relrelbar
}
\providecommand*{\xrightrightarrows}[2][]{%
	\ext@arrow 0359\rightrightarrowsfill@{#1}{#2}%
}
\providecommand*{\xleftleftarrows}[2][]{%
	\ext@arrow 3095\leftleftarrowsfill@{#1}{#2}%
}
\makeatother

\makeatletter
\newcommand*\Lslash
{%
	\text{\rlap{\kern.03em\@xxxii}}L%
}
\makeatother

\newcommand{\cosimp}[3]{\xymatrix@1{#1 \ar@<.4ex>[r] \ar@<-.4ex>[r] & {\ }#2 \ar@<0.8ex>[r] \ar[r] \ar@<-.8ex>[r] & {\ } #3 \ar@<1.2ex>[r] \ar@<.4ex>[r] \ar@<-.4ex>[r] \ar@<-1.2ex>[r] & \cdots }}
\newsavebox{\pullback}
\sbox\pullback{%
	\begin{tikzpicture}%
	\draw (0,0) -- (1ex,0ex);%
	\draw (1ex,0ex) -- (1ex,1ex);%
	\end{tikzpicture}}

\makeatletter
\newcommand*\dotp{\mathpalette\dotp@{.5}}
\newcommand*\dotp@[2]{\mathbin{\vcenter{\hbox{\scalebox{#2}{$\m@th#1\bullet$}}}}}
\makeatother

\newcommand{\equalizer}[2]{\xymatrix@1{#1 \ar@<.4ex>[r] \ar@<-0.4ex>[r] & {\ } #2}}

\newcommand{\adjunction}[4]{\xymatrix@1{#1{\ } \ar@<-0.3ex>[r]_{ {\scriptstyle #2}} & {\ } #3 \ar@<-0.3ex>[l]_{ {\scriptstyle #4}}}}

\usepackage{epigraph}

\usepackage{xcolor}
\usepackage{csquotes}
\usepackage{lipsum}

\definecolor{quotemark}{gray}{0.7}
\makeatletter
\newlength\origparskip

\newcommand{\fquote}{%
	\@ifnextchar[{\fquote@i}{\fquote@i[]}
}

\def\fquote@i[#1]{%
	\@ifnextchar[{\fquote@ii{#1}}{\fquote@ii{#1}[]}
}%

\def\fquote@ii#1[#2]{%
	\def\pqm@tempa{#1}%
	\def\pqm@tempb{#2}%
	\noindent
	\list
	{}
	{\setlength{\leftmargin}{0.3\textwidth}%
		\setlength{\rightmargin}{0.1\textwidth}%
		\setlength{\origparskip}{\parskip}}%
	\item[]%
	\begin{picture}(0,0)%
	\put(-15,-8){\makebox(0,0){\scalebox{4}{%
				\textcolor{quotemark}{\textquotedblright}}}}%
	\end{picture}%
	\begingroup
	\itshape
	\ignorespaces}%

\def\endfquote{%
	\endgroup
	\par
	\raggedleft
	\ifx\pqm@tempa\empty
	\else
	{\bfseries --- \pqm@tempa\par}%
	\setlength{\parskip}{\origparskip}%
	\ifx\pqm@tempb\empty
	\else
	(\pqm@tempb)%
	\fi
	\fi
	\par
	\endlist}
\makeatother

\begin{document}
	\bibliographystyle{alpha}
	\newtheorem{theorem}{Theorem}[section]
	\newtheorem*{theorem*}{Theorem}
	\newtheorem*{condition*}{Condition}
	\newtheorem*{definition*}{Definition}
	\newtheorem*{corollary*}{Corollary}
	\newtheorem{proposition}[theorem]{Proposition}
	\newtheorem{lemma}[theorem]{Lemma}
	\newtheorem{corollary}[theorem]{Corollary}
	\newtheorem{claim}[theorem]{Claim}
	\newtheorem{conclusion}[theorem]{Conclusion}
	\newtheorem{hypothesis}[theorem]{Hypothesis}
	\newtheorem{conjecture}[theorem]{Conjecture}
	\newtheorem{setup}[theorem]{Setup}
	\newtheorem{sumthm}[theorem]{Summary Theorem}

	\newtheorem{maintheorem}{Theorem}
	\renewcommand*{\themaintheorem}{\Alph{maintheorem}}
	\newtheorem{mainprop}[maintheorem]{Proposition}
	
	\theoremstyle{definition}
	\newtheorem{definition}[theorem]{Definition}
	\newtheorem{question}[theorem]{Question}
	\newtheorem{action}[theorem]{Action Item}
	\newtheorem{answer}[theorem]{Answer}
	\newtheorem{goal}[theorem]{Goal}
	\newtheorem{exercise}[theorem]{Exercise}
	\newtheorem{remark}[theorem]{Remark}
	\newtheorem{observation}[theorem]{Observation}
	\newtheorem{discussion}[theorem]{Discussion}
	\newtheorem{guess}[theorem]{Guess}
	\newtheorem{example}[theorem]{Example}
	\newtheorem{condition}[theorem]{Condition}
	\newtheorem{warning}[theorem]{Warning}
	\newtheorem{notation}[theorem]{Notation}
	\newtheorem{construction}[theorem]{Construction}
	
	\newtheorem{problem}{Problem}
	\newtheorem{fact}[theorem]{Fact}
	\newtheorem{thesis}[theorem]{Thesis}
	\newtheorem{convention}[theorem]{Convention}
	\newtheorem{summary}[theorem]{Summary}

	\title{The Gamified Kat\v{e}tov Order is Not Linear (in fact, very much not so)} 
    \author{Takayuki Kihara and Ming Ng}
    \thanks{TK was partially supported by JSPS KAKENHI Grant Numbers 23K28036 and 26K06892. MN was partially supported by a JSPS PostDoc Fellowship (Short-Term). } 
\maketitle

\begin{abstract} Recently, the authors introduced the Gamified Kat\v{e}tov order on filters over $\w$. This was shown to be strictly coarser than the classical Kat\v{e}tov order, and in fact collapses all MAD families to a single equivalence class. In the opposite direction, the present paper shows that the Gamified Kat\v{e}tov order also embeds $\mathcal{P}(\omega)/\mathrm{Fin}$, and thus contains an antichain of size continuum. The analysis brings into focus some interesting connections with Ramsey theory. As part of a broader programme investigating the interplay between combinatorial and computable complexity, we then apply our construction to produce a large new family of non-modest degrees in the extended Weihrauch hierarchy, which arise from associated effective subtoposes.
	
	
\end{abstract}

The Lawvere--Tierney order ($\clt$) in the Effective Topos, introduced by Hyland \cite{HylandEffective}, is a large-scale classification programme for investigating differences in computable complexity. Early work by Hyland and Pitts \cite{HylandEffective,PittsPhD} showed that the Turing degrees embed effectively into the $\clt$-order, and also established the existence of a minimum and maximum class. 
Attention therefore turned to investigating what lay in between the two extreme cases, a longstanding open problem in the area:
$$\id \sclt \,\,\cdots\cdots \,? \,\cdots\cdots \,\, \sclt \dn\,\,.$$

Very recently, we made the surprising discovery \cite{KiNg26} that this $\clt$-order is in fact tightly controlled by the combinatorics of filters on $\w$. To make the connection precise, we introduced what we call the {\em Gamified Kat\v{e}tov order} on filters over $\w$, before showing that a computable (and extended) variant of this order is isomorphic to the original $\clt$-order. This marks an important shift in our understanding. Category theorists have long attributed the difficulty of the $\clt$-order to the fact that the Effective Topos is not a Grothendieck topos. Our result locates a subtler issue: the structure of the $\clt$-order reflects deep shifts in set-theoretic complexity, now also interwined with constraints arising from computability.
\smallskip 

Read in the above context, the Gamified Kat\v{e}tov order can be seen as isolating the combinatorial mechanism underlying the $\clt$-order. Accordingly, we write $\glt$ for this order, suggestively echoing the notation $\clt$ with an added “$\mathsf{o}$”. However, the $\glt$-order also exhibits many striking features that make it interesting in its own right. Preliminary evidence indicates that it detects a form of complexity fundamentally distinct from those captured by usual set-theoretic tools: the $\glt$-order is strictly coarser than the classical Rudin-Keisler and Kat\v{e}tov orders \cite[Theorem A]{KiNg26}, and is also incomparable with the Tukey order on filters over $\w$ \cite[Theorem B]{KiNg26}.

The nature of its coarseness is also interesting. In the classical Kat\v{e}tov order, there exists a chain of $\mathfrak{c}^+$-many equivalence classes of MAD families as well as $\mathfrak{c}$-many pairwise incomparable ones \cite[\S2]{HruGF03}. By contrast, in the gamified context, all MAD families collapse to a single $\elt$-equivalence class \cite[Conclusion 5.3]{KiNg26}. On the other hand, we also discovered an infinite strictly ascending chain of classes within the $\glt$-order \cite[Theorem C]{KiNg26}. One reading of this result: despite its apparent coarseness, there still exists a rich internal structure within the Gamified Kat\v{e}tov order. 

\smallskip

The present paper extends this perspective. It would be consistent with previous results to conjecture that the $\glt$-order defines a countable linear order on the equivalence classes of filters (dually, ideals) on $\w$. However, we now show that this is emphatically not the case, answering \cite[Problem 3]{KiNg26}. Our first main theorem is the following.

\begin{maintheorem}[Theorem~\ref{thm:Pw-emb}]\label{mthm:Pwfin} The Gamified Kat\v{e}tov order admits an embedding of $\Pw/\Fin$. In particular, it contains a chain of length $\mathfrak{b}$, as well as an antichain of size $\mathfrak{c}$.
\end{maintheorem}

Two remarks regarding the implications of this result. First, as an immediate consequence, we also obtain an embedding of $\Pw/\Fin$ into the original $\clt$-order. From the category theory perspective, this identifies a large new family of effective subtoposes. From the viewpoint of computability theory, the embedding also yields a rich collection of non-modest degrees (in the context of Bauer's extended Weihrauch reducibility \cite{Bau22}). At this time of writing, non-modest degrees are still a relatively poorly understood class of computability notions; perhaps a large new supply of examples may begin to illuminate their structure. For more details, see Section~\ref{sec:topos}.

\medskip 
Second, the proof of Theorem~\ref{mthm:Pwfin} proceeds by embedding $\Pw/\Fin$ into the family of summable ideals.\footnote{A small caveat: the embedding sends the maximum class in $\Pw/\Fin$ to the subset family $\Pw$, which is technically not a summable ideal in the usual sense. However, see Theorem~\ref{thm:Pw-emb-Denz}.} This suggests asking how the Gamified Kat\v{e}tov order behaves {\em outside} the region of summable ideals. In particular, are there further instances of non-linearity to be found there?

\smallskip 

The following two theorems develop the beginnings of a bigger picture, identifying two additional instances of non-linearity.

\begin{maintheorem}[Theorem~\ref{thm:Sumn-Hind}]\label{mthm:Hindman} The canonical summable ideal $\Sumn$ and Hindman ideal $\calH$ are incomparable in the Gamified Kat\v{e}tov order. That is,
	$$\Sumn\not \glt \calH\qquad\text{and}\qquad \calH\not\glt \Sumn\,\,.$$
\end{maintheorem}

\begin{maintheorem}[Theorem~\ref{thm:Sumn-Ram}]\label{mthm:Ramsey} The canonical summable ideal $\Sumn$ and the Ramsey ideal $\Ram$ are incomparable in the Gamified Kat\v{e}tov order. That is,
	$$\Sumn\not \glt \Ram\qquad\text{and}\qquad \Ram\not\glt \Sumn\,\,.$$
\end{maintheorem}

Formal definitions will be deferred till Section~\ref{sec:other-ideals}. At a high level, Theorems~\ref{mthm:Hindman} and \ref{mthm:Ramsey} assert that ``Ramsey-like'' ideals such as $\calH$ and $\Ram$ look very different from the canonical summable ideal $\Sumn$. Given the coarseness of the Gamified Kat\v{e}tov order, this difference should be seen as significant. In fact, our proofs isolate specific structural properties to which $\glt$ is sensitive:
\begin{itemize}
	\item Earlier results (Theorem~\ref{mthm:Pwfin}, \cite[Theorem C]{KiNg26}) examined ideals defined by imposing specific asymptotic constraints on infinite subsets of $\w$. By contrast, the Hindman and Ramsey ideals are defined by forbidding certain kinds of highly structured patterns in infinite sets, which have no {\em a priori} connection to asymptotic growth. Section~\ref{sec:pos-dif-ideal} makes this distinction precise, and leverages it to show $\calH,\Ram\not\glt\Sumn$.

	\smallskip 
	
	\item The proof of Theorem~\ref{mthm:Pwfin} relies heavily on the pigeonhole principle to organise the relevant combinatorics. This invites the question of whether Ramsey-type theorems can also be used to separate ideal classes in $\glt$. Sections~\ref{sec:Hindman} and~\ref{sec:Ram} answer this affirmatively, using the Canonical Hindman and Ramsey theorems to prove $\Sumn\not\glt\calH,\Ram$.
\end{itemize}
Looked at in this light, our results highlight some suggestive connections with Ramsey theory that may point the way to future theorems. Theorems~\ref{mthm:Hindman} and \ref{mthm:Ramsey} suggest that partition theorems are effective tools for locating points of separation within the Gamified Kat\v{e}tov order. What other Ramsey-theoretic constructions or principles admit meaningful applications in our setting? 
\tableofcontents
 
 \section{Conventions and Preliminaries}
 
 \subsection{Notation and Conventions} 
 
 \begin{convention} 	$\w$ denotes the set of natural numbers, and $\Pw$ the powerset of natural numbers.
 \end{convention}
 
 \begin{convention}[Strings] $\ww$ denotes the set of infinite strings in $\w$, and $\lww$ the set of finite strings. For any string $x=(x_0,x_1,\dots)$, denote its {\em prefix of length $n$} as $x\rstr n:=(x_0,\dots,x_b)$. In particular, $\lww$ has a natural partial order $\preceq$ where $\tau\preceq \sigma$ just in case $\tau$ is a prefix of $\sigma$. Given any $\sigma,\tau\in\lww$, denote their concatenation as $\sigma\fr \tau$. The empty string is denoted $\o$.  
 \end{convention}
 
 \begin{convention}[Tree] A {\em tree} is a subset $T\subseteq \lww$ whereby:
 	\begin{enumerate}[label=(\alph*)]
 		\item $T$ contains the empty string $\o$, the {\em root} of $T$. 
 		\item $T$ is downward-closed under initial segments: if $\tau \preceq \sigma \in T$, then $\tau \in T$.
 	\end{enumerate}	
 We often call an element in $T$ a {\em node}.  An {\em infinite path through $T$} is a string $p \in \ww$ such that $p \rstr n \in T$ for all $n$. We write $[T]$ for the set of all infinite paths through $T$. 
 \end{convention}
 
 \begin{convention}[Partial Functions]
 	The notation $f\pcolon X\to Y$ denotes $f$ is a partial function from $X$ to $Y$, and the notation $f\pcolon X\tto Y$ denotes $f$ is a partial multi-valued function from $X$ to $Y$, i.e., $f\pcolon X\to\mathcal{P}(Y)$.
 	We often write $f(x)\dar$ if $x\in \dom(f)$ and $f(x)\uar$ otherwise;
 \end{convention}
 
 \begin{convention}[Partial Continuous Functions]\label{con:pcf} 	 
 	A {\em partial continuous function} $$\Phi\pcolon \ww \to \w$$
 	is determined by a partial function $\varphi\pcolon \lww \to \w$ satisfying:
 	\begin{enumerate}[label=(\alph*)]
 		\item {\em Extension.}
 		$$p \in \dom(\Phi) \,\, \text{and} \,\, \Phi(p) = c \iff \exists n\in\w \,\,\text{s.t.}\, p \rstr n \in \dom(\varphi) \,\,\text{and}\,\, \varphi(p \rstr n) = c.$$
 		\item {\em Coherence.} If $\sigma,\tau\in\dom(\varphi)$ and $\sigma\preceq \tau$, then $\varphi(\sigma)=\varphi(\tau).$ 
 	\end{enumerate}
 \end{convention}

\begin{convention}\label{conv:enum} Throughout this paper, we fix two different enumerations.
\begin{enumerate}[label=(\roman*)]
	\item {\em Ordered pairs.} Fix a computable bijection 
	$$	\lranglet{-}{-}\colon \w\times \w \simeq  \w\,,$$
	which enumerates by diagonals $(i,b)$ of constant sum $i+b=n$, increasing in $b$. In other words:
	\begin{itemize}
		\item Sum 0: $(0,0)$.
		\item Sum 1: $(1,0)$, $(0,1)$.
		\item Sum 2: $(2,0)$, $(1,1)$, $(0,2)$.
	\end{itemize}
etc. This particular enumeration has the following monotonicity property:
\smallskip
\begin{itemize}
	\item[] \begin{itemize}
		\item[$(\dagger)$]  For each fixed $i$: $b'<b \iff \lranglet{i}{b'}<\lranglet{i}{b}$.
	\end{itemize}
\end{itemize}
This will be helpful in streamlining the diagonalisation arguments of Section~\ref{sec:other-ideals}.
\smallskip
	\item {\em Unordered pairs.} Fix a computable bijection 
	$$\code{-}\colon [\w]^2\simeq \w\,$$
	where $[\w]^2$ denotes the set of distinct unordered pairs of elements in $\w$; here {\em distinct} means we exclude pairs of the form $\{a,a\}$. No additional properties on $\code{-}$ are imposed.
\end{enumerate}	
\end{convention}

 \begin{convention}[Filters and Ideals]\label{con:filters} Any collection of subsets $\calA\subseteq \Pw$ is called a {\em subset family}, and will be denoted in caligraphic font. In particular:
 	\begin{itemize}
 		\item An {\em upper set} $\calU \subseteq \Pw$ is a subset family upward closed under $\subseteq$, i.e. if $A \in \calU$ and $A \subseteq B$, then $B \in \calU$. 
 		A lower set is defined dually.
 		\item  For a lower set $\calI$, its {\em dual upper set} is $\calI^*:= \{A \subseteq \w : \w \setminus A \in \calI\}$; the dual lower set is defined analogously. To improve readability, we sometimes use parentheses and write $(\calI)^*$.
 		\item A {\em filter} is an upper set closed under finite intersections; dually, an {\em ideal} is a lower set closed under finite unions. 
 	\end{itemize}	
 \end{convention}

 \begin{example} $\Fin$ denotes the ideal of finite subsets of $\w$; $\Fin^*$ denotes the set of all cofinite subsets. 
 \end{example}

 \begin{definition}\label{def:finite-error} This paper will be interested in the combinatorics of sets and functions, up to some finite error. Some key definitions:
\begin{itemize}
	\item For $P,Q\subseteq\omega$, we write $P \subseteq^* Q$ whenever $|P\setminus Q|<\infty$ ; informally, $P$ is {\em almost included} in $Q$.
	\item If $P\subseteq^* Q$ and $Q\subseteq^* P$, we write $P=^*Q$ ; informally, $P$ is {\em almost equal} to $Q$.
	\item For functions $f,g\colon \omega\to [0,\infty)$, we write $f\leq^* g$ if $f(n)\leq g(n)$ for all but finitely many $n$.
\end{itemize}
These notions naturally lead to the quotient structure
\[
(\mathcal P(\omega)/\mathrm{Fin},\,\subseteq^*),
\]
where $\mathcal P(\omega)/\mathrm{Fin}$ denotes the powerset of $\omega$ modulo $\mathrm{Fin}$, the finite set ideal. The relation $\subseteq^*$ induces the natural partial order on this quotient.
 \end{definition}

 \subsection{The Gamified Kat\v{e}tov order} The Gamified Kat\v{e}tov order is a game-theoretic refinement of the classical Kat\v{e}tov order, introduced by the authors in \cite[\S 3]{KiNg26}. We start by reviewing the usual definition of Kat\v{e}tov reduction. Given any $\calU,\calV\subseteq\Pw$ upper sets, we write
 	$$\calU\lk\calV \,$$
 just in case	
 	$$\,\exists \,h\colon\w\to\w\,\text{such that}\, A\in\calU \implies h^{-1}[A]\in\calV\,.$$
Informally, $\calU,\calV$ each represent two abstract notions of largeness, and $\calU\lk\calV$ means that every $\calU$-large set $A$ can be uniformly witnessed to be $\calV$-large by a single fixed map $h\colon\w\to\w$. A useful reformulation, exploiting the fact that  $\calU,\calV$ are upper sets, is that $\calU\lk\calV$ if and only if 
$$\,\exists \,h\colon\w\to\w\,\text{such that}\, A\in\calU \implies h[B]\subseteq A\,\,\text{for some}\, B\in\calV\,;$$
see \cite[Observation 3.7]{KiNg26}. In plainer terms: instead of pulling $A$ back along $h$, one may equivalently push forward some $\calV$-large set $B$ into $A$. 

\smallskip

The Gamified Kat\v{e}tov order generalises this picture on two levels. 
\begin{enumerate}
	\item It replaces the witnessing map $h\colon\w\to\w$ with a partial map on finite strings $$\varphi\pcolon \lww\to\w.$$ Thinking of a finite string as a sequence of values revealed one at a time suggests a sequential, in particular a game-like, interpretation of the construction (which we do not formalise here; for details, see \cite[\S 3.4]{KiNg26}).
	\item The single set $B\in\calV$ is replaced by a tree whose branching is controlled by $\calV$. From the game-theoretic perspective, the tree is a device that supplies a  $B_n\in\calV$ at every stage $n$ of the game.
\end{enumerate}
In this way, a single global witness $B$ is replaced by a family of stage-by-stage witnesses organised into a tree, while the map $\varphi$ assigns an output to each resulting sequence. This framework is formalised by the following key definitions.

\begin{definition}\label{def:game-tree} Let $\calV\subseteq\Pw$ be an upper set, and $T\subseteq\lww$ a tree. A node $\sigma\in  T$ is called {\em $\calV$-branching} if the set of immediate successors $\{n\in\w\mid \sigma\fr n\in T\}$ is in $\calV$. A tree $T$ is {\em $\calV$-branching} just in case: 
	\begin{enumerate}[label=(\alph*)]
		\item $T$ is non-empty; and
		\item For any $\sigma\in T$, $\sigma$ is $\calV$-branching.
	\end{enumerate}
	We emphasise Condition (b) includes the case when $\sigma=\o$, the root of $T$.
\end{definition}

\begin{definition}[Gamified Kat\v{e}tov order]\label{def:GTK} Let $\calU,\calV\subseteq\Pw$ be upper sets. We write
	$$\calU\glt\calV$$
just in case there exists a partial continuous function $$\Phi\pcolon \ww\to\w$$
	such that, for any $A\in\calU$, there exists some $\calV$-branching tree $T$ satisfying
	$$	[T]\subseteq \dom(\Phi)
	\qquad\text{and}\qquad
	\Phi[T]\subseteq A.$$
Equivalently, $\Phi(p)\in A$ for any infinite path $p$ through $T$.
\end{definition}

\begin{convention} We typically say $\Phi$ {\em witnesses} the reduction $\calU\glt\calV$ whenever it satisfies the conditions of Definition~\ref{def:GTK}, and call any such $\Phi$ a {\em witness}.
\end{convention}

\begin{convention} By dualising, one obtains a corresponding notion on lower sets: for $\calH,\calI\subseteq\Pw$,
	$$\calH \glt \calI
	\quad\text{iff}\quad
	\calH^* \glt \calI^*.$$
	
\end{convention}

As alluded to before, the original motivation behind the Gamified Kat\v{e}tov order came not from set theory, but category theory. More precisely, it arises from the study of the  Effective Topos $\Eff$, introduced by Hyland \cite{HylandEffective}, which provides a category-theoretic framework for computability.

Every (elementary) topos carries a natural partial order on its Lawvere--Tierney topologies (the Lawvere-Tierney order). In well-behaved cases, such as Grothendieck toposes, the $\clt$-order admits several useful concrete descriptions, and topos theorists have developed a wide range of tools to analyse them. By contrast, the $\clt$-order in $\Eff$ is far less understood. While we know it encodes computability-theoretic information --- for instance, it admits an effective embedding of the Turing degrees \cite[Theorem 17.2]{HylandEffective} --- a systematic analysis of its structure has proved difficult.

\smallskip

A key move in \cite{KiNg26} is to identify the combinatorial mechanism underlying these category-theoretic structures. Re-examining work of Lee-van Oosten \cite[\S 2]{LvO13}, we observed that the $\clt$-order in $\Eff$ induces a preorder on upper sets in $\Pw$. This set up the first major theorem of our paper \cite{KiNg26}:

 \begin{theorem}[{{\cite[Theorem A]{KiNg26}}}]\label{thm:main-thm}  The Gamified Kat\v{e}tov order ($\glt$) defines a preorder on upper sets over $\w$. Moreover:
 	\begin{enumerate}[label=(\roman*)]
 		\item The \emph{Gamified Kat\v{e}tov order} is equivalent to the {\em Kat\v{e}tov order} closed under {\em well-founded iterations of Fubini powers}. In particular, it is strictly coarser than the Kat\v{e}tov order.
 		\item The {\em Gamified Kat\v{e}tov order} admits an explicit game-theoretic description, justifying its name.
 		\item The {\em computable Gamified Kat\v{e}tov order} is equivalent to the {\em $\clt$-order} on upper sets over $\w$.
 	\end{enumerate}
\end{theorem}
 
 \begin{discussion}[Coarseness] For further details, we refer the reader to \cite{KiNg26}. The main takeaway here is that Theorem~\ref{thm:main-thm} locates two distinct sources of coarseness within the Gamified Kat\v{e}tov order.
 \begin{itemize}
 	\item By (i), the order is closed under well-founded Fubini iterations. In particular, $\calU \elt \calU \otimes \calU$ for every upper set $\calU$ -- which is emphatically false for the classical Kat\v{e}tov order. Indeed, many separation results for the Kat\v{e}tov order on ideals rely on the strict inequality $\Fin<_{\mathrm{K}}\Fin\otimes \Fin$, and so this collapse creates serious technical obstacles in the gamified setting; see Discussion~\ref{dis:pos-dif-Fin}.
 	\item By (iii), the $\clt$-order is equivalent to the Gamified Kat\v{e}tov order once we restrict to computable witnesses $\Phi$. Hence, the $\glt$-order is also coarser than the original $\clt$-order -- in particular, $\calU\not\glt\calV$ implies $\calU\not\clt\calV$. Showing that $\glt$ is in fact {\em strictly} coarser takes work, and leads to other subtleties -- see Corollary~\ref{cor:no-Denz} and Problem~\ref{prob:Pw-LT}.
 \end{itemize}

 \end{discussion}

 \subsection{Separation Lemmas} This subsection explains the separation technique developed in \cite[\S 5]{KiNg26}. Throughout, we let $\calU,\calV\subseteq\Pw$ be upper sets. To begin, recall that a set $A\subseteq \w$ is called  {\em $\calU$-null} just in case $\w\setminus A\in\calU$; otherwise, $A$ is called {\em $\calU$-positive}. Informally, a set $A$ is $\calU$-positive if the upper set $\calU$ regards $A$ as being {\em non-negligible} in size -- notice this does not always imply $A\in\calU$. The following example, though elementary, will be important for us:

\begin{example} Let $\calI$ be an ideal, with dual filter $\calI^*$. Then, any $A\in\calI$ is $\calI^*$-null and any $A\notin\calI$ is $\calI^*$-positive. 
\end{example}

 \begin{observation}\label{obs:positive-null} Let $\calU\subseteq\Pw$ be an upper set. If $A$ is $\calU$-positive and $B\in\calU$, then $A\cap B\neq\emptyset.$
 \end{observation}
 \begin{proof} If $A\cap B=\emptyset$, then $B\subseteq \w\setminus A$. Since $B\in\calU$ and $\calU$ is upward closed, this implies $\w\setminus A\in\calU$. In other words, $A$ is $\calU$-null, a contradiction.
 \end{proof}

 Let us revisit Definition~\ref{def:GTK}. A witness $\Phi$ to $\calU\glt\calV$ is, in general, only a \emph{partial} function on $\lww$. For the purposes of separation, it is useful to extend $\Phi$ to a total map on $\lww$ in a way that reflects the largeness notion encoded by $\calU$. Here is the basic idea.  For any node in $\lww$, whenever $\calU$ identifies a non-negligible set of successors carrying a common value, we assign that value to the node itself; otherwise, we assign the node the value $\bot$. In this way, $\calU$ acts as a notion of “majority” guiding how undefined values are filled in. This leads to the technical definition of a {\em labelling function}.
 
 \begin{definition}[Labelling Function]\label{def:label} Let $\Phi\pcolon\ww\to\w$ be a partial continuous function, and $\calU$ an upper set. Then:
 	\begin{enumerate}[label=(\roman*)]
 		\item Define the {\em canonical tree} $T_\Phi$ as the $\preceq$-downward closure of $\{\sigma\in\lww \mid \Phi(\sigma)\dar\}$.
 		\item Define a {\em labelling function} 
 		$$\nu\colon \lww\to\w\cup\{\bot\}$$
 		recursively as follows:
 		\begin{enumerate}[label=(\arabic*)]
 			\item If $\Phi(\sigma)\dar$, then set $\nu(\sigma)=\Phi(\sigma)$.
 			\item If $\sigma\notin T_\Phi$, then set $\nu(\sigma)=\bot$.
 		\end{enumerate}

 		\begin{enumerate}[label=(\arabic*)]
 			\setcounter{enumii}{2} 
 			\item If there exists $c\in \w$ such that $\{n\in\w\mid \nu(\sigma\fr n)=c\}$  is $\calU$-positive for some $c\in\w$, then set $\nu(\sigma)=c$. If there are multiple options, choose the least $c$.
 			\item Otherwise, set $\nu(\sigma)=\bot$.
 		\end{enumerate}
 		\item For each finite string $\sigma$, its {\em $\nu$-label} is the value $\nu(\sigma)\in\w\cup\{\bot\}$.
 		\item A node $\sigma\in\lww$ is {\em $\nu$-critical} if $\nu(\sigma)=\bot$ and $\{n\mid \nu (\sigma\fr n)=\bot\}$ is $\calU$-null; a {\em $\nu$-critical successor} is a successor of a critical node.
 	\end{enumerate}	
 \end{definition}

  \begin{convention} Let $\Phi$ be a partial continuous function and $\calI$ an ideal. To emphasise dependence, we sometimes write $\nu_\Phi^{\calI}$ to mean a labelling function 
 	defined using the dual filter of $\calI$. However, whenever the context is clear, we shall suppress the indices for readability and simply write $\nu$. Similarly, we typically say that a node is ``critical'' or speak of its ``label'', instead of writing out ``$\nu$-critical'' or ``$\nu$-label'' in full.
 \end{convention}

 The following figure illustrates the labelling process.
 
\begin{figure}[h!]
	\includegraphics[width=12cm]{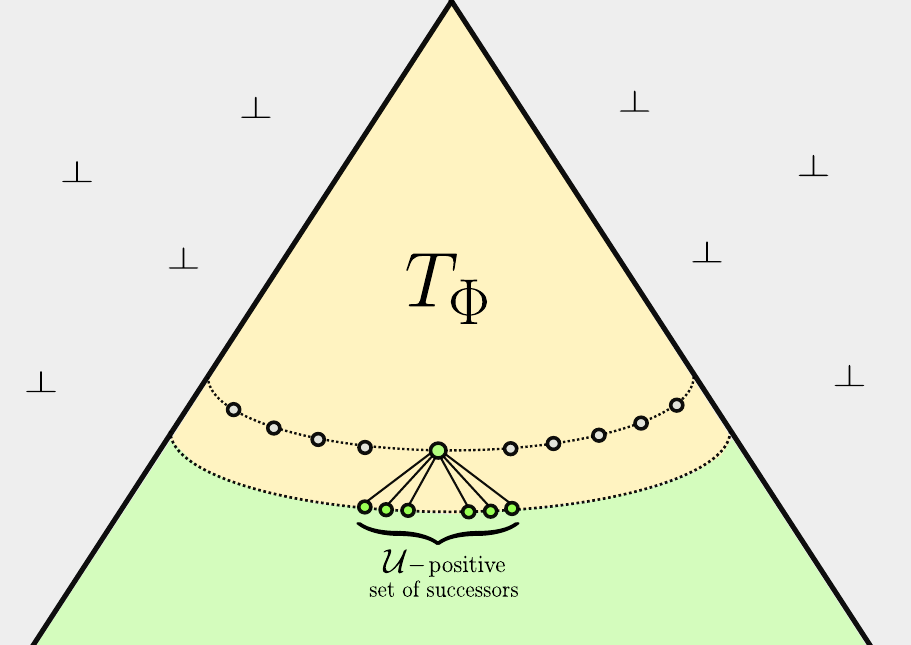}
	\caption{The grey shaded region lies outside the canonical tree $T_\Phi$; here, all nodes are labelled $\bot$. Within $T_\Phi$, the values assigned by $\Phi$ propagate upward along $\calU$-positive sets of successors with the same label. Nodes where no such agreement occurs are also labelled $\bot$. }
	\label{fig:labeltree}
\end{figure}

 Finally, we explain the terminology ``canonical tree''. Suppose $\Phi$ witnesses $\calU\glt\calV$. By definition, for each $A\in\calU$ there exists a $\calV$-branching tree $T_A$ such that $\Phi[T_A]\subseteq A$.  Now suppose $T_\Phi$ is the canonical tree associated to $(\Phi,\calV)$, with labelling function $\nu:=\nu^\calV_\Phi$. By construction, 
$\nu$ is a total function extending the original $\Phi$, and $T_\Phi$ is the $\preceq$-downward closure of $\dom (\Phi)$.
 Hence,
  $$T_A\subseteq T_\Phi\qquad\text{and}\qquad \nu (p)=\Phi(p)\,\,\text{for all}\, p\in [T_A]\,.$$ 
 In this sense, $T_\Phi$ is a universal object: it simultaneously contains all potential witness trees $T_A$ for $\calU\glt\calV$. This means its structure imposes specific constraints on how values of $\Phi$ may be distributed along any $\calV$-branching tree. The following two lemmas identify two such constraints, which will play a crucial role in our analysis.

 \begin{lemma}[{{\cite[Lemma 5.10]{KiNg26}}}]\label{lem:separation-1} Let $T$ be a $\calV$-branching tree.
 	If a node $\sigma\in T$ is $\nu^\calV_\Phi$-labeled by $c\in\w$, then there exists an infinite path $p$ through $T$ such that $\Phi(p)=c$.
 \end{lemma}

 \begin{lemma}[{{\cite[Lemma 5.11]{KiNg26}}}]\label{lem:separation-2} Let $T$ be a $\calV$-branching tree. If $[T]\subseteq\dom(\Phi)$ and its root is labeled by $\bot$, then $T$ has a $\nu^\calV_\Phi$-critical node.
 \end{lemma}
 
 \section{Embedding $\Pw/\Fin$}\label{sec:Pw}

In this section, we prove Theorem~\ref{mthm:Pwfin}: the Gamified Kat\v{e}tov order embeds a copy of $\Pw/\Fin$. The core ideas are the same as in \cite[Theorem 1]{GGMA16} due to Meza-Alc\'{a}ntara, 
 which embeds a copy of $\Pw/\Fin$ into the Kat\v{e}tov order on ideals. However, due to the coarseness of $\glt$, more work is needed to separate the ideal classes. 
\medskip 

To begin, recall that an ideal $\calI$ is called {\em summable} just in case there exists a function $f\colon \w\to [0,\infty)$ satisfying $\lim_{m\to\infty}f(m)=0$, $\sum_{m\in\w}f(m)=\infty$ and 
$$\calI=\Sumf:=\bigg\{A\subseteq \w \,\,\bigg|\,\, \sum_{m\in A} f(m)<\infty \bigg\}\,.$$
The condition $\sum_{m\in\w}f(m)=\infty$ ensures that $\Sumf$ defines a {\em proper} ideal; otherwise, $\Sumf=\Pw$, which, per the conventions of this paper, we still regard as an ideal (albeit a degenerate one).
\smallskip

We now lay the groundwork for the proof. The general idea is to assign an ideal $\Sump$ to each $P\subseteq \w$, and then verify that 
\begin{equation}\label{eq:equiv-Pw-fin}
P\subseteq^* Q \iff \Sump\glt \Sumq\,.
\end{equation}
We use the same construction as in \cite{GGMA16}.

\begin{construction}\label{cons:Pw-fin} Proceed in stages.
	\begin{itemize}
		\item 	Fix a partition $\{I_{n}\}_{n\in\w}$ of $\w$ into finite intervals, and a descending sequence $\{r_n\}_{n\in\w}$ of positive rational numbers satisfying the conditions:\footnote{For an explicit construction of $\{I_{n}\}_{n\in\w}, \{r_n\}_{n\in\w}$, see \cite[Theorem 1]{GGMA16}. 
		}

		\begin{enumerate} 
			\item $|\bigcup_{j<n}I_{j}|\leq r_{n} |I_n|$ for all $n\geq 1$; 
			\item $|I_n|r_{n+1}\leq 2^{-n-1}$.
			\item $I_0=\{0\}$ and $r_0=1$.
		\end{enumerate}
	
	\smallskip
		\noindent Hereafter, we write $I_{<n}:=\bigcup_{j<n}I_{j}$, and $I_{\geq n}:=\bigcup_{j\geq n} I_j$.
		
		\smallskip
		\item For each set $S\subseteq\w$, define its {\em weight function} $w_S\colon\w\to [0,\infty)$ as follows:
		$$w_S(m)=\begin{cases}
		r_n \qquad\,\,\,\,\,\,\text{if}\,m\in I_n\,\text{and}\, n\notin S\\
		r_{n+1} \qquad\text{if}\,m\in I_n\,\text{and}\, n\in S\\
		\end{cases}\,\,.$$ 	
		This extends to a map $w_S\colon \Pw\to [0,\infty]$ whereby $w_S(A):=\sum_{n\in A}w_S(n)$.
		
		\smallskip
\item For each $S\subseteq\w$, define the ideal
$$\mathrm{Sum}_S:=\mathrm{Sum}_{w_S}\,.$$
	\end{itemize}
\end{construction}

The next two observations explain the conditions placed within Construction~\ref{cons:Pw-fin}. First and foremost, Conditions (1) and (2) guarantee the following:

\begin{observation}\label{obs:co-infinite} For every co-infinite $S\subseteq\w$, $\mathrm{Sum}_S$ defines a summable ideal. 
\end{observation}
\begin{proof} By Condition (2), it is clear $\lim_{n\to\infty} r_n = 0$, and so  $\lim_{m\to\infty}w_S(m)=0$. Next, notice
	$$\sum_{m\in\w} w_S(m) = \sum_{n\in\w} \sum_{m\in I_n} w_S(m)\geq \underbrace{\sum_{n\in\w\setminus S } r_n|I_n| \geq \sum_{n\in\w\setminus S }|I_{<n}| }_{\text{Condition (1)}} = \infty\,,$$
where the final RHS sum diverges since $S$ is assumed to be co-infinite. 
\end{proof}

%

Moreover, Condition (2) also yields the following degenerate case:

\begin{observation}\label{obs:w} If $S=\w$, then $\mathrm{Sum}_S=\Pw$. 
\end{observation}
\begin{proof} Since $S=\w$, compute
	$$\sum_{m\in\w} w_S(m)=\underbrace{\sum_{n\in\w} |I_n| r_{n+1}\leq \sum_{n\in\w} 2^{-n-1}}_{\text{Condition (2)}} <\infty\,,$$
and so every subset of $\w$ has finite $w_S$-weight. Hence, conclude that $\mathrm{Sum}_S=\Pw$.
\end{proof}

More generally, if $S\subseteq\omega$ fails to be co-infinite, then $S$ must be
cofinite, and so $S=^*\omega$. Reviewing the definition of $\mathrm{Sum}_S$, this means all but finitely many intervals $I_n$ receive the smaller weight $r_{n+1}$, and so the total $w_S$-weight remains bounded. In which case, the ideal construction again collapses to yield $\mathrm{Sum}_S=\Pw$. As it turns out, this phenomenon follows from a more general fact: 

\begin{claim}\label{claim:Pw-first-direction} Suppose $P,Q\subseteq\w$ and $P\subseteq^* Q$. \underline{Then}, $\Sump\subseteq \Sumq$, and so $\Sump\glt\Sumq$.
\end{claim}
\begin{proof} 
If $P\subseteq^* Q$, then $w_Q\leq^* w_P$, and so $\Sump\subseteq \Sumq$. To show this implies $\Sump\glt \Sumq$, consider the partial continuous function $\varphi\pcolon\lww\to\w$ induced by the identity map $\id\colon\w\to\w$. Since $\Sump\subseteq \Sumq$, any subset $A\in(\Sump)^*$ defines a $(\Sumq)^*$-branching tree [of height 1], and so the claim follows. 
\end{proof}

Notice that Claim~\ref{claim:Pw-first-direction} also proves the forward direction of the desired equivalence~\eqref{eq:equiv-Pw-fin}. It remains to prove its contrapositive:

\begin{claim}\label{claim:Pw-second-direction} Suppose $P,Q\subseteq\w$ and $P\not\subseteq^* Q$. \underline{Then}, $\Sump\not\glt\Sumq$.
	
\end{claim}
\begin{proof} Suppose for contradiction there exists a partial continuous function $\Phi$ witnessing $\Sump\glt\Sumq$. Denote the labelling function to be $\nu:=\nu_\Phi^{\Sumq}$, and construct the canonical tree $T_\Phi$. There are two main cases to check.
	
\subsubsection*{Case 1: $\nu(\o)\in\w$.} Suppose the root of $T_\Phi$ is labelled $c\in\w$. Clearly, $\{c\}\in \Sump$. Hence, denoting $A:=\w\setminus \{c\}$, we have that $A$ belongs to $(\Sump)^*$, the dual filter of $\Sump$. By assumption, there exists a $(\Sumq)^*$-branching tree $T_A$ such that $[T_A]\subseteq\dom(\Phi)$ and $\Phi[T_A]\subseteq A$. However, by Lemma~\ref{lem:separation-1}, there exists $p\in [T_A]$ such that $\Phi(p)=c\notin A$, a contradiction.		
	
\subsubsection*{Case 2: $\nu(\o)\in\bot$.} By Lemma~\ref{lem:separation-2} and our assumption that $\Phi$ witnesses $\Sump\glt\Sumq$, $T_\Phi$ contains at least one $\nu$-critical node in $T_\Phi$. Hence, enumerate all such critical nodes as $\{\tau_{i}\}_{i\in\w}$. [By allowing repetitions, we may assume the enumerating index set is $\w$.] 

\smallskip 
	

For each critical node $\tau_i$ and each interval $I_n$ of $\w$, we can divide $I_n$ using three different colours based on the associated label of its elements:
\begin{align}\label{eq:coarse-label}
\widetilde{\nu}_{i,n}\colon I_n&\longrightarrow 3  \\
x &\longmapsto \begin{cases}
0\qquad\text{if}\, \nu(\tau_i\fr x)=\bot\\
1 \qquad\text{if}\, \nu(\tau_i\fr x)\in I_{<n} \\
3 \qquad\text{if}\, \nu(\tau_i\fr x)\in I_{\geq n}\\
\end{cases}\quad . \nonumber
\end{align}
We then use this colouring to define some key parameters. 
\begin{itemize}
	\item For every $n\in\w$, there exists a subset 
	$F_n\subseteq I_n$ and some $t_n<3$ such that 
	\begin{equation}\label{eq:tn-bound}
	|F_n|\geq \frac{1}{3}|I_n|\,\,\text{and}\,\, \widetilde{\nu}_{i,n}(x)=t_n \,\,\text{for all}\, x\in F_n. 
	\end{equation}
	Moreover, if $n\notin Q$ and $n\neq 0$, then Condition (1) yields
	\begin{equation}\label{eq:one-third-bound}
	w_Q(F_n)\geq\frac{r_n}{3}|I_n|
	\geq \frac{1}{3}|I_{<n}|
	\geq \frac{1}{3}\, .
	\end{equation}
	If $n\notin Q$ and $n=0$, then Condition (3) guarantees the same inequality since $\frac{r_0}{3}|I_0|=\frac{1}{3}$.  
	\item Since $P\not\subseteq^*Q$ by hypothesis, the set $P\setminus Q$ is infinite. Hence, for every fixed critical node $\tau_i$, we obtain infinitely many indices $n\in P\setminus Q$, each with an associated 
	value $t_n\in\{0,1,2\}$. Since there are only three possible values, the 
	pigeonhole principle implies there exists some $t<3$ such that 
	\[
	t_n=t
	\]
	for infinitely many $n\in P\setminus Q$. Our analysis now splits into three further sub-cases, based on the value of $t$. 
\end{itemize} 

\smallskip  

To orient the reader, here is the informal proof strategy for Case 2. Working over {\em all} critical nodes $\{\tau_i\}_{i\in\w}$, we aim to extract a ``bad'' collection $\{C_k\}_{k\in\w}$ of their successor labels. We will then use this collection to construct a set $A\in (\Sump)^*$, and show that any $\Phi$ witnessing $\Sump\glt\Sumq$ cannot witness $A$. From this contradiction, we conclude $\Sump\not\glt\Sumq$. However, to define this bad collection first requires a case-by-case analysis 
of how the successor labels are distributed, which we now turn to:
\medskip 

\begin{itemize}
	\item \textbf{Case 2a: $t=0$ .}  Let $H_i$ be the infinite subset of $n$'s in $P\setminus Q$ such that $t_n=0$. Since $\tau_i$ is critical, 
	$$F:=\bigcup_{n\in H_i} F_n\subseteq \{x\mid \nu(\tau_i\fr x)=\bot\}\in\Sumq \,,$$
	and so $w_Q(F)<\infty$ by definition. However, since $n\notin Q$ for all $n\in H_i$, applying Equation~\eqref{eq:one-third-bound} yields
	$$w_Q(F)=\sum_{n\in H_i} w_Q(F_n)\geq \sum_{n\in H_i}\frac{1}{3}=\infty,$$
	a contradiction. 
	So this case never happens.

	\smallskip 
	\item \textbf{Case 2b: $t=1$ .} Let $H_i$ be the infinite subset of $n$'s in $P\setminus Q$ such that $t_n=1$. By definition, for each $n\in H_i$, we have $\nu(\tau_i\fr x)\in I_{<n}$ for every $x\in F_n$. By the finite pigeonhole principle, one can find $G_n\subseteq F_n$ and some fixed  $c^i_n\in I_{<n}$ such that $|G_n|\geq |F_n|/|I_{<n}|$ and $\nu(\tau_i\fr x) = c^i_n$ for any $x\in G_n$. In particular, by Equation~\eqref{eq:tn-bound} and Condition (1), this yields the inequality
$$	|G_n|\geq \frac{|F_n|}{|I_{<n}|}\geq  \frac{|I_n|}{3|I_{<n}|}  \geq \frac{1}{3 r_n} \,\,.$$
	Moreover, since $n\notin Q$, the $w_Q$-weight of each $x\in G_n$ is $r_n$, and so we obtain \begin{equation}\label{eq:Gn-third}
	w_Q(G_n)\geq \frac{1}{3}\,\,,
	\end{equation}
	and so for $G:=\bigcup_{n\in H_i} G_n$, we obtain the total weight of
\begin{equation}\label{eq:wQ-G}
w_Q(G)=\infty\,.
\end{equation}

Collecting all $\nu$-labels, suppose the set of labels $K:=\{c^i_n\mid n\in H_i\}$ is finite. For each $k\in K$, define $G[k]:=\{x\in G\mid \nu(\tau_i\fr x)=k\}$. Clearly, $G=\bigcup_{k\in K} G[k]$. In which case, by Equation~\eqref{eq:wQ-G}, we have
$$\sum_{k\in K}w_Q (G[k])=w_Q(G)=\infty\,,$$
and so by the finiteness of $K$, there must exist some $k$ such that $w_Q(G[k])=\infty$. However, as $\tau_i$ is critical, this means $$G[k]\subseteq \{n\mid \nu(\tau_i\fr n)=k\}\in \Sumq,$$
and so $G[k]$ must have finite $w_Q$-weight, a contradiction. Hence, the set $K:=\{c^i_n\mid n\in H_i\}$ must be infinite for all critical nodes. 	
	\smallskip 
\item \textbf{Case 2c: $t=2$ .} Let $H_i$ be the infinite subset of $n$'s in $P\setminus Q$ such that $t_n=3$. By definition, for each $n\in H_i$ and every $x\in F_n$, we have $c_x:=\nu(\tau_i\fr x)\in I_{\geq n}$, and so $w_P(c_x)\leq r_{n+1}$. In fact, collecting all $\nu$-labels, set $E_n:=\{c_x\mid x\in F_n\}$. Applying Condition (2) then yields
	$$ w_P(E_n)=\sum_{x\in F_n} w_P(c_x)\leq r_{n+1}|F_n|\leq r_{n+1}|I_n|\leq 2^{-n-1}\,.$$

\end{itemize}
		
We now construct three sequences inductively:
\begin{itemize}
\item A sparse sequence $\{C_k\}_{k\in\w}$ of successor labels;
	\item A sequence of subsets $\{D[i]\}_{i\in\w}$ with sufficient $w_Q$-weight;
	\item A sequence of finite subsets $\{V_k\}_{k\in\w}$, recording the previously used indices from $P\setminus Q$.
\end{itemize}
This is arranged by a standard diagonalisation over all $k=\lranglet{i}{b}$, where we visit each critical node $\{\tau_i\}_{i\in\w}$ infinitely often, and sample (distinct) successors sufficiently far out along $\w$ to ensure sparseness.

At stage $k=\lranglet{i}{b}$, we select a successor label of $\tau_i$. 
\begin{itemize}
	\item Suppose Case 2b holds for this $i$. Since $K:=\{c^i_n\mid n\in H_i\}$ is infinite, there exists $n\in H_i$ such that $\lranglet{i}{n}\notin V_{k-1}$ and $c^i_n\notin I_{<k}$.  Hence, set $$C_k:=\{c^i_n\}\,\quad  V_k:=V_{k-1}\cup \{\lranglet{i}{n}\}\quad\text{and}\quad D_k:=G_n.$$
	Notice this yields the bounds $$w_P(C_k)\leq r_{k}\leq 2^{-k}\quad \text{and}\quad w_Q(D_k)\geq \frac{1}{3};$$
	 this follows from Condition (2) and Equation~\eqref{eq:Gn-third}.
	\item Suppose Case 2c holds for this $i$. Since $H_i$ is infinite, there exists $n\in H_i$ such that $n>k$ and $\lranglet{i}{n}\notin V_{k-1}$. Hence, set 
	$$C_k:=E_x=\{c_x \mid x\in F_n\}\,\quad  V_k:=V_{k-1}\cup \{\lranglet{i}{n}\}\quad\text{and}\quad D_k:=F_n \,.$$
Notice this yields the same bounds as in Case 2b. Since $n>k$, we have $w_P(C_k)\leq 2^{-n-1}\leq 2^{-k}$; the fact that $w_Q(D_k)=w_Q(F_n)\geq \frac{1}{3}$ follows from Equation~\eqref{eq:one-third-bound}.

\end{itemize}
Notice since $\lranglet{i}{n}\notin V_{k-1}$ at every stage, we avoid double-counting: in particular, $D_{\lranglet{i}{b}}\cap D_{\lranglet{i}{b'}}=\emptyset$ for any $b\neq b'$. Notice also that $\nu(\tau_i\fr x)\in C_k$ holds for any $x\in D_k$.
\smallskip

We now have all the necessary pieces to assemble the main argument. Define
$$D[i]:=\bigcup_{b\in\w} D_{\lranglet{i}{b}}\,.$$
Notice for every $k=\lranglet{i}{b}$, we have $D_{\lranglet{i}{b}}\subseteq I_n$ for a distinct $n$, and $w_Q(D_{\lranglet{i}{b}})\geq \frac{1}{3}$. Hence, $$w_Q(D[i])=\infty,$$
and so $D[i]$ is $(\Sumq)^*$-positive. Moreover, collecting all associated labels 
$$C:=\bigcup_{k\in\w} C_k,$$
we get $w_P(C)=\sum_{k\in\w}w_P(C_k)\leq \sum_{k\in\w} 2^{-k}<\infty$. Hence, $C\in \Sump$, and so if we take $A:=\w\setminus C$, then $A\in (\Sump)^*$.
\smallskip

By assumption, there exists a $(\Sumq)^*$-branching tree $T_A$ such that $[T_A]\subseteq \dom(\Phi)$ and $\Phi[T_A]\subseteq A$. By Lemma~\ref{lem:separation-2}, $T_A$ has a critical node $\sigma=\tau_i$. Since $\tau_i\in T_A$ and $T_A$ is $(\Sumq)^*$-branching, this means $$\{n\mid \tau_i\fr n\in T_A\}\in (\Sumq)^*\,.$$
Since $D[i]$ is $(\Sumq)^*$-positive, by Observation~\ref{obs:positive-null} there exists some $x\in D[i]$ such that $\tau_i\fr x\in T_A$. Moreover, for any $x\in D[i]$, there exists $k=\lranglet{i}{b}$ such that $x\in D_k$, which is arranged to have the label $c:=\nu(\tau_i\fr x)\in C_k\subseteq C$ by construction. Applying Lemma~\ref{lem:separation-1}, this means there must exist $p\in [T_A]$ such that $\Phi(p)=c\notin A$, a contradiction. This completes the proof.
\end{proof}

\begin{theorem}\label{thm:Pw-emb} There exists an order embedding of $\big(\Pw/\Fin,\subseteq^*\!\!\big)$ into the $\glt$-order on ideals over $\w$. In particular, it contains a chain of length $\mathfrak{b}$, and antichains of size $\mathfrak{c}$.
\end{theorem}
\begin{proof} The embedding is immediate from Claims~\ref{claim:Pw-first-direction} and~\ref{claim:Pw-second-direction}. The rest follows from well-known properties of $\Pw/\Fin$.
\end{proof}

\begin{conclusion} Despite its coarseness, the Gamified Kat\v{e}tov order is highly non-linear -- in particular, it admits continuum-sized antichains of pairwise incomparable classes.
\end{conclusion}

It is reasonable to read Theorem~\ref{thm:Pw-emb} as embedding a copy of $\Pw/\Fin$ into the family of summable ideals. However, there is a caveat. By Observation~\ref{obs:w}, the maximum element $[\w]\in \Pw/\Fin$ corresponds to $\mathrm{Sum}_\w=\Pw$, which strictly speaking is not a summable ideal. One solution would be to expand the original definition of a summable ideal to include non-proper ideals. In our view, a more natural approach is to identify sharper bounds for $\mathrm{Sum}_S$ within $\glt$ .

\begin{theorem}\label{thm:Pw-emb-Denz}  The {\em asymptotic zero density ideal} is defined as
	$$\Denz:= \Big\{A\subseteq \w \Bmid \lim_{n\to\infty} \frac{|A\cap [0,n]|}{n}=0\Big\}\,.$$
The following hold:
	\begin{enumerate}[label=(\roman*)]
		\item For any co-infinite $S\subseteq\w$, 
		$$\Sums\sglt \Denz\,.$$
		\item For any infinite $S\subseteq\w$, 
		$$\Fin\sglt \Sums\,.$$
	\end{enumerate} 
Consequently, the Gamified Kat\v{e}tov order embeds a copy of $\Pw/\Fin$ between $\Fin$ and $\Denz$. 
\end{theorem}	
\begin{proof} \hfill \begin{enumerate}[label=(\roman*):]
		\item For any summable ideal $\Sumf$, it is known that $\Sumf\lk \Denz$ \cite[Proposition 3.5]{HHH07}. Hence, by Theorem~\ref{thm:main-thm} and Observation~\ref{obs:co-infinite}, this means $\Sums\glt \Denz$ for all co-infinite $S\subseteq\w$. Moreover, an easy check shows there does not exist a maximal co-infinite set $M$ whereby $A\subseteq^* M$ for all co-infinite sets $A$. 
		 Combining Claims~\ref{claim:Pw-first-direction} and \ref{claim:Pw-second-direction}, deduce that
		 $$\Sums\sglt \Denz.$$
		 \item Clearly, $\Fin\subseteq \Sums$ for any $S\subseteq\w$ and so $\Fin\glt\Sums$. To show the reduction is strict whenever $S\subseteq\w$ is infinite, apply Claims~\ref{claim:Pw-first-direction} and \ref{claim:Pw-second-direction}, which separates $\Sump$ and $\Sums$ whenever $P\subseteq\w$ is finite. 
	\end{enumerate}
Hence, combining items (i), (ii) and  Theorem~\ref{thm:Pw-emb}, this yields an embedding of $\Pw/\Fin$ into $\glt$ between $\Fin$ and $\Denz$, now assigning $\mathrm{Sum}_{\w}$ to $\Denz$ instead of $\Pw$, and $\Sump$ to $\Fin$ for any finite $P$.
\end{proof}

\begin{discussion} A natural next question is how the summable ideals arising in Theorem~\ref{thm:Pw-emb} compare with the so-called {\em canonical summable ideal} 
$$\Sumn:=\Big\{A\subseteq \w \Bmid \sum_{n\in A}\frac{1}{n+1}<\infty\Big\}\,.$$
At present, no complete comparison is known. Informally, the difficulty lies in comparing the harmonic weight function $1/(n+1)$ (which decays smoothly to 0), and the weight function $w_S$ from Construction~\ref{cons:Pw-fin} (which also tends to $0$ but is constant on long intervals). Investigating these differences may also illuminate other questions about $\Sumn$. For instance, if there exists a co-infinite $S\subseteq\w$ such that 
$$\Sumn \glt \Sums\,,$$
then this would imply $\Sumn\sglt\Denz$, thereby answering a question in \cite[Problem 6]{KiNg26}.
\end{discussion}

\section{Ramsey-Like Ideals}\label{sec:other-ideals}

Given any partial order $\leq$ on ideals (or filters), proving separation results of the form 
$$\calI_0<\calI_1\,,$$
requires establishing a positive direction (i.e. $\calI_0\leq \calI_1$) and a negative one
(i.e. $\calI_1\not\leq \calI_0$). Typically, the negative direction is harder: the positive direction only requires exhibiting {\em some} witness to the reduction $\calI_0\leq \calI_1$ whereas the negative direction requires ruling out {\em every} possible witness for  $\calI_1\leq \calI_0$. 

In this section, we re-examine the main mechanism underlying the negative arguments
from before, and extend it to identify further instances of non-linearity in the
Gamified Kat\v{e}tov order.

\subsection{A Blueprint}\label{sec:blueprint} We begin by reviewing, in broad strokes, the proof strategy of
Claim~\ref{claim:Pw-second-direction}, which showed that $\Sump\not\glt\Sumq$ whenever $P\not\subseteq^* Q.$
\begin{enumerate}
	\item Start by \textbf{partitioning} $\w$ into intervals $\{I_n\}_{n\in\w}$, which are then used to construct an ideal $\Sums$ for every $S\subseteq\w$.
\item Suppose for contradiction that $P\not\subseteq^*Q $ and there exists a $\Phi$ witnessing $\Sump\glt \Sumq$. Applying Definition~\ref{def:label}, construct the canonical tree $T_\Phi$.
\item Verify that the root of $T_\Phi$ cannot be labelled $c\in\w$, so assume the root is labelled $\bot$. Informally, $\nu(\o)=\bot$ indicates a breakdown of coherence during the recursive labelling of $T_\Phi$'s nodes. More precisely, $T_\Phi$ contains at least one critical node $\tau_i$.
\item For each critical node $\tau_i$, examine how the labels of its successors are distributed. Since $P\not\subseteq^*Q$ by hypothesis, $P\setminus Q$ is \textbf{large}. By repeated applications of the \textbf{pigeonhole principle}, we extract, for each $b<\w$, a \textbf{relatively large} set of successors $D_{\lranglet{i}{b}}$ that is \textbf{homogeneous} in the following sense.

\smallskip 
\begin{itemize}
	\item[] 
	Using the interval partition, decompose $\w\cup\{\bot\}$ into three pieces, and define a colouring function
	$$\widetilde{\nu}\pcolon \lww \to 3\,,$$
obtained by recording which piece the original label $\nu(\sigma)$ belongs to. A set $D$ of successors of $\tau_i$ is \emph{homogeneous} if $\widetilde{\nu}$ is constant on $\{\tau_i \fr x \mid x \in D\}$.  Thus, for each fixed $b$, every $x\in D_{\lranglet{i}{b}}$ receive the same colour
	$$\widetilde{c}_{\lranglet{i}{b}} := \widetilde{\nu}(\tau_i \fr x).$$ 
\end{itemize}

\item Collect the (original) successor labels into the set
$$C:=\bigcup_{i\in\w}\big\{\nu(\tau_i\fr x )\bmid x\in D_{\lranglet{i}{b}}\big\}.$$
Check that $C$ is suitably small, i.e. $C\in \Sump$.

Next, fixing each $\tau_i$, take the union of the corresponding homogeneous sets
$$ D[i]=\bigcup_{b\in\w} D_{\lranglet{i}{b}}.$$
Then, check that $D[i]$ is of non-negligible size, i.e. $D[i]$ is  $(\Sumq)^*$-positive.

\item Since $C$ is small, the complement $$A:=\omega\setminus C$$
belongs to $(\Sump)^*$ and is therefore witnessed by $\Phi$ (by hypothesis). That is, there exists a suitable tree $T_A$ such that $\Phi[T_A]\subseteq A$. However, the size of the $D[i]$'s prevent this: any candidate witness tree $T_A$ can be shown to intersect some $D[i]$. This forces the tree to include a label from $C$, and thus $\Phi[T_A]\not\subseteq A$.

Informally, whenever $P\setminus Q$ is large, we can construct a family $\{D[i]\}_{i\in\w}$ of  obstructions: each $D[i]$ is sufficiently large to create unwanted collisions with the forbidden labels in $C$, and thus prevent $A$ from being consistently witnessed.

\item Given the contradiction in (6), conclude that $\Sump\not\glt\Sumq$.
\end{enumerate}

\medskip 

The appearance of the pigeonhole principle here will be suggestive to any reader familiar with Ramsey theory. In particular, a key move in Step (4) involves passing from a fine labelling $\nu$ to a finite coarse colouring $\widetilde{\nu}$, before extracting large homogeneous sets. One may therefore ask: by varying the partition structure of $\w$, or suitably modifying our informal notions of ``large'' and ``homogeneous'', might we discover other instances of non-linearity within the gamified setting? This basic question guides our next move.

\subsection{Ramsey and Hindman Ideals} Various Ramsey-like ideals have already been studied in the classical Kat\v{e}tov order \cite{FKK24,FKK24-spaces}. Here we identify two such ideals, before discussing their associated partition theorems. 

\begin{convention} Let $A\subseteq\w$ be any subset.
\begin{itemize}
	\item $[A]^2$ denotes the set of all {\em distinct unordered pairs} of elements in $A$, $[A]^{<\w}$ denotes the family of all {\em finite subsets} of $A$, and $[A]^\w$ denotes the family of all {\em infinite subsets} of $A$.
	\item Define
	$$\FS(A):= \Big\{ \sum_{n\in F} n \Bmid F\in [A]^{<\w}\setminus \{\emptyset\}\Big\}\,.$$
	That is, $\FS(A)$ is the set of all finite [non-empty] sums of distinct elements in $A$. 
\end{itemize}		
\end{convention}

\subsubsection*{Ramsey Ideal} The {\em Ramsey ideal} is defined as 
$$\Ram:=\Big\{ A\subseteq [\w]^2\Bmid \forall B\in [\w]^\w.\,([B]^2\not\subseteq A)\Big \}\,.$$
Here is one way to understand this ideal. Any set $A\subseteq [\w]^2$ defines a graph $G_A=(\w,A)$ where $\w$ denotes the vertices of $G_A$ and $A$ the set of edges.\footnote{Per our definition of $[A]^2$, this means we do not allow self-loops in our graph $G_A$.} In which case, $\Ram$ is the ideal consisting of graphs that do not have infinite complete subgraphs. That this defines an ideal is a consequence of Ramsey's Theorem:

\begin{theorem}[Canonical Ramsey Theorem, {{\cite[Theorem II]{ErRa50}}}]\label{thm:Ramsey} For every function $f\colon [\w]^2\to\w$, there exists an infinite set $T\subseteq  \w$ such that one of the four cases hold:
	\begin{enumerate}
		\item $\forall x,y \in [T]^2.\,\, f(x)=f(y)\,.$
\item $\forall x,y \in [T]^2.\,\, f(x)=f(y)\iff \min x = \min y.$
\item $\forall x,y \in [T]^2.\,\, f(x)=f(y)\iff \max x = \max y.$
\item $\forall x,y \in [T]^2.\,\, f(x)=f(y)\iff x = y.$
	\end{enumerate}
\end{theorem}

\begin{discussion} Any map $f\colon[\w]^2\to\w$ induces a partition of $[\w]^2$ via its fibres
	$$f^{-1}(\{n\})\qquad (n\in\w).$$
The Canonical Ramsey Theorem can therefore be viewed as a rigidity result: every such partition admits a large set $T\subseteq \w$ on which the partition map $f$  assumes a highly structured (canonical) form. There are four such cases: (1) $f$ is constant; (2) $f$ depends only on the minimum element; (3) $f$ depends only on the maximum element; or (4) $f$ is injective. 
\end{discussion}

\medskip 
\subsubsection*{Hindman Ideal} The {\em Hindman ideal} is defined as
$$\calH:=\Big\{ A\subseteq \w \Bmid \forall B\in [\w]^\w.\,(\FS(B)\not\subseteq A)\Big \}\,.$$
That is, $\calH$ is the ideal consisting of subsets $A\subseteq\w$ that always omit some finite sum of {\em any} infinite set. This, too, has an associated partition theorem. Define $E:=\{2^n\mid n\in\w\}$. For any $x\in\FS(E)$, define its {\em support} as the (unique) subset $\alpha(x)\subseteq E$ such that
$$x=\sum\alpha(x)\,.$$

\begin{theorem}[Canonical Hindman Theorem, {{\cite[Theorem 2.1]{Tay76}}}]\label{thm:Hindman} For every function $f\colon\w\to\w$ there exists an infinite set $H=\{h_n\mid n\in\w\}\subseteq\w$ such that
	$$\max\alpha(h_n)<\min \alpha (h_{n+1}) \,\,\forall n\in\w\,,$$
and one of the following five cases holds:
	\begin{enumerate}
		\item $\forall x,y\in \FS(H).\, f(x)=f(y)$,
		\item $\forall x,y\in \FS(H).\, f(x)=f(y)\iff \min\alpha(x)=\min \alpha(y)$,
			\item $\forall x,y\in \FS(H).\, f(x)=f(y)\iff \max\alpha(x)=\max \alpha(y)$,
\item $\forall x,y\in \FS(H).\, f(x)=f(y)\iff \min\alpha(x)=\min \alpha(y)$ {\em and} $\max\alpha(x)=\max \alpha(y)$,
		\item $\forall x,y\in \FS(H).\, f(x)=f(y)\iff x=y$ .
	\end{enumerate}
\end{theorem}

\begin{discussion} Each $h_n\in H$ corresponds to a finite subset $\alpha(h_n)\subseteq E$ whereby $h_n=\sum \alpha(h_n)$. The condition that $\max \alpha(h_n)<\min \alpha(h_{n+1})$ tells us that the $\alpha(h_n)$'s are disjoint for distinct $n$'s, and are thus relatively spaced out across $\w$. This will play a key role in our proof of Theorem~\ref{thm:Sumn-Hind}.
\end{discussion}

\begin{discussion} The original  Canonical Hindman's Theorem is usually stated differently \cite{Tay76}: given 
	$$f\colon [\w]^{<\w}\to \w,$$ there exists a {\em disjoint collection} $$\calD:=\{d_n \in [\w]^{<\w}\mid d_i\cap d_j=\emptyset\, ,\,\,\text{whenever}\, i\neq j\}$$ such that $f$ assumes one of five canonical forms on $s,t\in \mathrm{FU}(\calD)$.\footnote{Here $\mathrm{FU}(\calD)$ denotes the family of all finite unions of elements of $\calD$ (excluding the empty union).} Our Theorem~\ref{thm:Hindman} reparametrises this by encoding finite subsets of $\w$ through binary expansion using the set
	$E:=\{2^n\mid n\in\w\}.$
In fact, any subset $E\subseteq\w$ that is {\em sparse} in the technical sense of \cite[pp.~862]{FKK24} yields the same reformulation, but we fix this particular choice for concreteness.
	
\end{discussion}

\subsection{Positive Difference Ideal}\label{sec:pos-dif-ideal}  We will eventually want to establish various non-linearity results involving $\Ram$ and $\calH$. As a warm-up, first define
\begin{align}\label{eq:pos-dif}
\Delta\colon[\w]^\w&\longrightarrow [\w]^\w\\
B&\longmapsto \{a-b\mid a,b\in B\,, a>b\,\}. \nonumber 
\end{align}
That is, $\Delta(B)$ is the set of all positive differences of distinct elements in $B$. This naturally leads to the {\em positive difference ideal}, defined as
$$\calD:=\Big\{ A\subseteq \w \Bmid \forall B\in [\w]^\w.\,(\Delta(B)\not\subseteq A)\Big \}\,.$$
That $\calD$ indeed defines an ideal follows from Ramsey's Theorem, see \cite[Proposition 4.1]{Fi13}. 

\begin{observation}\label{obs:difference-ideal} $\Fin\subsetneq \calD\subsetneq\calH$. In particular, $\calD$ contains all modular sets
	$$A_m:=\{n\in\w\mid m\nmid n\}\quad (m\geq 2)\,.$$
\end{observation}
\begin{proof} It is clear that $\Fin\subseteq \calD\subseteq\calH$. Next, fix any modular set $A_m$ for some $m$. Given any infinite $B\subseteq\w$, there exists two elements in $B$ which are congruent mod $m$ by the pigeonhole principle. So their difference is divisible by $m$ -- which is excluded from $A_m$. Hence, $\Delta(B)\not\subseteq A_m$ for every infinite $B$. Since $A_m$ is infinite, this implies $\Fin\subsetneq \calD$. To see why $\calD\subsetneq\calH$, see \cite[Proposition 4.2]{Fi13}.
\end{proof}

Observation~\ref{obs:difference-ideal} highlights a key detail: $\calD$ collects subsets $A\subseteq\w$ that avoid certain kinds of structured differences that inevitably emerge in any infinite subset. This is quite a different organising principle from collecting subsets whose tail ends have bounded weight, the structural limitation built into the definition of summable ideals. 
Our present goal is to make this difference precise. 


\smallskip

We start by introducing a useful combinatorial notion: a subset $A\subseteq\w$ is {\em eventually $\calD$-sparse} if for every $d\in\w$ there exists at most finitely many pairs $(m,n)\in A^2$ with $m>n$ such that $d=m-n$. In particular:

\begin{lemma}\label{lem:sparse} If $A\subseteq\w$ is eventually $\calD$-sparse, then $A\in\calD$.
\end{lemma}
\begin{proof} Suppose for contradiction $A$ is eventually $\calD$-sparse but $A\notin\calD$. Then, there exists an infinite $E\subseteq\w$ with $\Delta(E)\subseteq A$. Pick $b<c\in E$, and let 
	$$\delta:=c-b\,.$$
Now take infinitely many $d_i\in E$ with $d_i>c$ and strictly increasing. For each $i$, define
$$x_i:= d_i-b\,,\,\, y_i:=d_i-c\,.$$	
Then $x_i,y_i\in\Delta(E)\subseteq A$, and 
$$x_i-y_i=(d_i-b)-(d_i-c)=c-b=\delta\,.$$
Hence, the pair $(x_i,y_i)\in A^2$ realises the same difference $\delta$ for every $i$. Moreover, all these pairs $(x_i,y_i)$ are distinct since $d_i$ is strictly increasing, and so $x_i=d_i-b$ is strictly increasing. This gives infinitely many distinct pairs in $A$ with the same difference, contradicting eventual $\calD$-sparseness.
\end{proof}

\begin{theorem}\label{thm:pos-dif} Recall the {\em canonical summable ideal}  $$\Sumn:=\Big\{A\subseteq \w \Bmid \sum_{n\in A}\frac{1}{n+1}<\infty\Big\}\,.$$ \underline{Then}, $\calD\not\glt \Sumn$.
	
\end{theorem}
\begin{proof} Suppose for contradiction there exists $\Phi$ witnessing $\calD\glt\Sumn$. Denote the labelling function as $\nu:=\nu_\Phi^{\Sumn}$. If the root of $T_\Phi$ is labelled by $c\in\w$, then we get a contradiction by the same argument in Claim~\ref{claim:Pw-second-direction}. Hence, as before, assume $\nu(\o)=\bot$ and enumerate the critical nodes $\{\tau_i\}_{i\in\w}$.
	
\smallskip		
To begin, notice each critical node $\tau_i$ must have infinite distinct $\nu$-labels amongst its successors. Why? Suppose for contradiction there exists only finitely many labels, say $c_0,\dots, c_m$. Since $\tau_i$ is critical, we have 
$$S_i:=\{x\mid \nu(\tau_i\fr x)=c_i\}\in\Sumn\qquad 0\leq i\leq m\,,$$
and so we have finite harmonic weight $\sum_{x\in S_i}\frac{1}{x+1}<\infty$ for each $i$. But since the $S_i$'s give a finite partition of $\w$, this implies 
$$\sum_{x\in\w}\frac{1}{x+1}= \sum_{x\in S_0}\frac{1}{x+1} + \sum_{x\in S_1}\frac{1}{x+1} + \dots  + \sum_{x\in S_m}\frac{1}{x+1} <\infty\,, $$
a contradiction. This sets up our diagonalisation scheme below.

\smallskip

\noindent \underline{Goals and Conditions.} At stage $k=\lranglet{i}{b}$ of the diagonalisation, we choose a subset of successors of $\tau_i$ with relatively large harmonic weight, yet whose labels are still sufficiently sparse. Specifically, this involves constructing three sequences
$$\{D_k\}_{k\in\w}\,\qquad \{C_k\}_{k\in\w}\, \qquad \,\{L_k\}_{k\in\w}\,$$
 whereby:
\begin{itemize}
\item $D_k$ denotes a chosen subset of successors of the critical node $\tau_i$ at stage $k$;
	\item $C_k$ denotes the corresponding set of labels, $C_k:=\{\nu(\tau_i\fr x )\mid x\in D_k\}\,;$
	\item $L_k$ denotes the set of previous label differences,  $$L_k:=\Delta (\bigcup_{j<k} C_j)\,.$$
\end{itemize}
These are chosen to satisfy the conditions:
\begin{enumerate}[label=(\alph*)]

	\item  {\em (Strong Sparseness).} $C_k$ is finite, and 
	$$(x,y)\in C_k\times \bigcup_{j\leq k} C_j \implies |x-y|\notin L_k\,,$$
i.e.	no new pair involving $C_k$ realises a difference in $L_k$. 

	\smallskip
	
	\item {\em (Well-labelled).} No successor in $D_k$ is labelled $\bot$; in particular, $C_k\subseteq \w$.
	
	\smallskip 
	
%
	
	\item {\em (Positivity).} $$\sum_{x\in D_k}\frac{1}{x+1}\geq 1\,\,\,.$$
	
\end{enumerate} 

\noindent \underline{Diagonalisation.} At stage $k=\lranglet{i}{b}$, suppose $D_0,\dots, D_{k-1}$ have been chosen. In which case, both $\bigcup_{j<k}C_j$ and $L_k$ are finite. This allows us to define the upper bounds 
$$ \n_k:=\max \bigcup_{j<k} C_j\qquad \text{and}\qquad \m_k:= \max L_k + 1\,.$$
In light of the previous label choices, we proceed to rule out certain successors of the critical node $\tau_i$. Start by defining
$$ S_k:=\Big\{x\in\w \Bmid \nu(\tau_i\fr x)\notin \{\bot\} \,
\,\text{and}\,\, \nu(\tau_i\fr x)>\n_k + \m_k \Big\}\,\qquad T_k:=\Big\{\,\nu(\tau_i\fr x) \Bmid x\in S_k\,\Big\}\,.$$
This imposes two conditions: (1) exclude successors of $\tau_i$ labelled $\bot$; and (2), the successor must have a label sufficiently bigger than $\n_k$. Notice this {\em a fortiori} implies that $\nu(\tau_i\fr x)\notin \bigcup_{j<k}C_j$. Moreover, notice $T_k$ defines an infinite set of labels -- every critical node $\tau_i$ has infinite distinct succesor labels, and the conditions in $S_k$ only rule out a finite subset of such labels. 
\smallskip

Next, an additional refinement. Start by partitioning $\w$ into residue classes $\mathrm{mod}$ $\m_k$:
$$\w=\bigsqcup_{r=0}^{\m_k-1} \m_k \N + r\,.$$
Defining 
$$T_{k,r}:=(\m_k\N + r )\cap T_k\,,\qquad\text{and}\qquad S_{k,r}:=\{x\in\w \mid \nu(\tau_i\fr x)\in T_{k,r}\}\,,$$
this induces a partition of $S_k$ and  $T_k$ as below
$$S_k=\bigsqcup_{r=0}^{\m_k-1} S_{k,r}\qquad\text{and}\qquad  T_k=\bigsqcup_{r=0}^{\m_k-1} T_{k,r}\,.\qquad $$
Notice all labels in each residue class $T_{k,r}$ are at least $\m_k$ apart, and at least $\m_k$-distance away from the biggest label $\n_k \in \bigcup_{j<k}C_j$ -- this will be important for guaranteeing strong sparseness. Moreover, we have the following key claim:

\begin{claim}\label{claim:pos-dif-harmonic} At least one $S_{k,r}$ must have infinite harmonic weight, i.e. 
	$$\sum_{x\in S_{k,r}}\frac{1}{x+1}=\infty\,.$$
\end{claim}
\begin{proof}[Proof of Claim] Suppose for contradiction 
	$$ \sum_{x\in S_{k,r}} \frac{1}{x+1} <\infty\qquad\text{for all}\, \,\,0\leq r \leq \m_k-1\,.$$
	This would imply
	$$\sum_{x\in S_{k}}\frac{1}{x+1}= \sum_{x\in S_{k,0}}\frac{1}{x+1} + \sum_{x\in S_{k,1}}\frac{1}{x+1} + \,\cdots\,  + \sum_{x\in S_{k,\m_k-1}}\frac{1}{x+1} <\infty\,.$$
	But this is clearly false -- the full successor set of $\tau_i$ has divergent harmonic sum, and $S_k$ has only removed finitely many successor labels, each of finite harmonic weight (since $\tau_i$ is critical). 
\end{proof}

We now finish our diagonalisation argument. By Claim~\ref{claim:pos-dif-harmonic}, at least one residue class $S_{k,r}$ must have infinite harmonic weight. Fix such an $r$. We may then choose a finite subset $D_{k}\subseteq S_{k,r}$ such that
 $$\sum_{x\in D_k} \frac{1}{x+1} \geq 1,$$
and define $C_k:=\{\nu(\tau_i\fr x) \mid x\in D_k\}$, clearly a finite set as well. Reviewing our requirements from before, one checks that our construction guarantees strong sparseness, positivity and well-labelled successors at every stage of the diagonalisation.

\medskip

\noindent \underline{Verification.} We have now done most of the combinatorial legwork; it remains to implement the proof strategy (``blueprint'') presented at the start of the section. \begin{itemize}
	\item Define
	$$D[i]:=\bigcup_{b\in\w} D_{\lranglet{i}{b}}.$$
	By positivity, each $D_{\lranglet{i}{b}}$ has harmonic weight at least 1 for all $b$. It is also clear $D_{\lranglet{i}{b}}$ are disjoint across different $b$'s since they correspond to different label sets. Hence,
	$$\sum_{x\in D[i]} \frac{1}{x+1} = \infty\,,$$
	and so $D[i]$ is $(\Sumn)^*$-positive. 
	\item Collect all the successor labels into the set
	$$C:=\bigcup_{k} C_k.$$
At each stage $k$, the label set $C_k$ is chosen so that no pair $$(x,y)\in C_k\times \bigcup_{j\leq k} C_j$$ 
realises a difference already present in the set of previous label differences
$$L_k = \Delta\Big(\bigcup_{j<k} C_j\Big).$$
Hence, no difference is ever realised in more than one stage of the diagonalisation. Moreover, since each $C_k$ is finite, it follows that every difference in $\Delta(C)$ is realised by only finitely many pairs. Hence, $C$ is eventually $\calD$-sparse, and so $C\in\calD$ by Lemma~\ref{lem:sparse}.

\end{itemize}

\medskip 

\noindent\underline{Contradiction.} Since $C\in\calD$, its complement $A:=\w\setminus C$ belongs to its dual filter $\calD^*$. Since $\Phi$ witnesses $\calD\glt\Sumn$ by hypothesis, there exists a $(\Sumn)^*$-branching tree $T_A$ such that $\Phi[T_A]\subseteq A$. By Lemma~\ref{lem:separation-2}, $T_A$ has a critical node $\sigma=\tau_i$. By definition, $\{n\mid \tau_i\fr n\in T_A\}\in (\Sumn)^*$ since $T_A$ is $(\Sumn)^*$-branching. Since $D[i]$ is $(\Sumn)^*$-positive, there exists some $x\in D[i]$ such that $\tau_i\fr x\in T_A$. Denote its label as $c:=\nu(\tau_i\fr x)$. By Lemma~\ref{lem:separation-1}, there must exist some $p\in [T_A]$ such that $\Phi(p)=c$. However, we have $c\in C$ by construction, and so $\Phi(p)=c\notin A$, contradicting  $\Phi[T_A]\subseteq A$.
\end{proof}

Our primary motivation behind establishing Theorem~\ref{thm:pos-dif} lies in its structural consequences for the Hindman and Ramsey ideals.

\begin{corollary}\label{cor:first-neg-H-R} The following negative results hold:
	\begin{enumerate}[label=(\roman*)]
		\item $\calH\not\glt \Sumn$.
		\item $\Ram\not\glt \Sumn$.
	\end{enumerate}	
\end{corollary}
\begin{proof} \hfill
\begin{enumerate}[label=(\roman*):]
	\item This is easy. Since $\calD\subseteq\calH$ (Observation~\ref{obs:difference-ideal}), we have $\calD\glt\calH$. Hence, if $\calH\glt\Sumn$, then 
	$$\calD\glt\calH\glt\Sumn\,,$$
	contradicting Theorem~\ref{thm:pos-dif} that $\calD\not\glt \Sumn$.	 
	\item \cite[Theorem 7.7 (4)]{FKK24-spaces} shows that $\calD\lk\Ram$ -- in which case, this implies $\calD\glt\Ram$ and the same argument as in item (i) shows that $\Ram\not\glt\Sumn$. 
	
	The original argument showing $\calD\lk\Ram$ makes use of so-called {\em partition regular functions} \cite[\S 3]{FKK24-spaces}. For completeness, we give a more direct proof. Define a function 
	\begin{align*}
	\varphi\colon [\w]^2 &\longrightarrow \w\\
\{a,b\}&\longmapsto |a-b|\,\,.  
	\end{align*}
Notice for any set $F\subseteq\w$, this gives
	$$\varphi\left[[F]^2\right]=\Delta (F).$$
	
Now suppose $A\subseteq [\w]^2$ such that $A\notin\Ram$. By definition, there exists some infinite $B\subseteq\w$ such that $[B]^2\subseteq A$. Clearly then
	$$\Delta(B)=\varphi\left[[B]^2\right]\subseteq \varphi[A],$$
	and so $\varphi[A]\notin\calD$. Hence, given any $X\in\calD$, it must be the case that $\varphi^{-1}[X]\in\Ram$; otherwise, $\varphi^{-1}[X]\notin\Ram$ implies $\varphi\left[\varphi^{-1}[X]\right]\notin\calD$. But $\varphi\left[\varphi^{-1}[X]\right]\subseteq X$, and since $\calD$ is downward closed, this contradicts $X\in\calD$.
\end{enumerate}	
\end{proof}

\begin{discussion}\label{dis:pos-dif-Fin} Theorem~\ref{thm:pos-dif} generalises \cite[Theorem 7.7 (11)]{FKK24-spaces}, which gave a quick proof showing  $\calD\not\lk\Sumn$ in the classical Kat\v{e}tov order. The original argument shows that $\Fin\otimes \Fin \lk \calD$ yet $\Fin\otimes\Fin \not\lk \Sumn$. In which case, $\calD\lk\Sumn$ implies $\Fin\otimes\Fin\lk\Sumn$, a contradiction. Notice, however, this approach fails in the gamified setting since $\Fin\elt \Fin\otimes \Fin$, and so a completely different proof strategy is required.

\end{discussion}

\begin{discussion}[``Sparseness'']	Our definition of ``eventually $\calD$-sparse'' generalises a stronger notion introduced by Filip\'{o}w \cite{Fi13}: a subset $A\subseteq\w$ is called {\em $\calD$-sparse} if every difference $d\in\Delta(A)$ is realised by a unique pair $b,c\in A$ with $d=b-c$. While $\calD$-sparseness also ensures $A\in\calD$, it is too restrictive for our purposes: enforcing global uniqueness forces excessive separation, preventing the accumulation of sufficient harmonic weight. Eventual $\calD$-sparseness relaxes this by allowing controlled repetition of differences.
\end{discussion}

\begin{discussion}[Partitions] Notice partition arguments still play a key role in organising the combinatorics -- especially when controlling the size of the set of forbidden labels $C$. In Theorem~\ref{thm:Pw-emb}, we partitioned $\w$ into large consecutive intervals $\{I_n\}_{n\in\w}$; in Theorem~\ref{thm:pos-dif}, we partition $\w$ into residue classes for some large $\m_k$ (at every stage $k$).
\end{discussion}

\subsection{$\Sumn$ and $\calH$ are incomparable}\label{sec:Hindman} We now identify another instance of non-linearity within the Gamified Kat\v{e}tov order. Our argument generalises a theorem of Filip\'{o}w and Kowitz \cite[Theorem 3.1]{FKK24}, who established the same result in the setting of the classical Kat\v{e}tov order.

\medskip

\begin{theorem}\label{thm:Sumn-Hind} $\Sumn$ and $\calH$ are incomparable in the Gamified Kat\v{e}tov order. That is,
	$$\Sumn\not \glt \calH\qquad\text{and}\qquad \calH\not\glt \Sumn\,\,.$$
\end{theorem}
\begin{proof} By Corollary~\ref{cor:first-neg-H-R}, we already know that $\calH\not\glt\Sumn$. It remains to show $\Sumn\not\glt \calH$.
	
\medskip	
	
Suppose for contradiction there exists a partial continuous function $\Phi$ witnessing $\Sumn\glt \calH$. Denote the labelling functon as $\nu:=\nu_\Phi^{\calH}$, and construct the canonical tree $T_\Phi$. By the same argument as in Claim~\ref{claim:Pw-second-direction}, the root of $T_\Phi$ cannot be labelled $c\in\w$, so assume $\nu(\o)=\bot$. By Lemma~\ref{lem:separation-2}, it follows that $T_\Phi$ contains at least one critical node. Enumerate all the critical nodes $\{\tau_i\}_{i\in\w}$. This sets up our diagonalisation scheme below.

\smallskip

\noindent \underline{Goal}. Our main objective is to construct a forbidden label set
$$C\in\Sumn\,,$$ such that for every critical node $\tau_i$, we obtain an $\calH^*$-positive set
$$D[i]\notin \calH\,,$$
constructed so that $\nu(\tau_i\fr x)\in C$ for all $x\in D[i]$.

\smallskip 

\noindent \underline{Partition.} For each critical node $\tau_i$, the labelling function $\nu$ induces a map
\begin{align}\label{eq:partition-map}
\nu_i\colon \w&\longrightarrow \w\\
x
& \longmapsto \begin{cases}
\lranglet{0}{0}\qquad\qquad\quad\,\,\,\,\text{if}\,\nu(\tau_i\fr x)=\bot\\
\lranglet{1}{\nu(\tau_i\fr x)} \qquad\quad \text{if}\,\nu(\tau_i\fr x)\in \w
\end{cases} \,\,. \nonumber
\end{align}
Applying the Canonical Hindman Theorem (Theorem ~\ref{thm:Hindman}), fix an infinite set
\begin{equation}\label{eq:H-Hindman}
H[i]=\{h_n\mid n\in\w\}\subseteq \w
\end{equation}
such that $\max \alpha(h_n)<\min \alpha (h_{n+1})$ for every $n$, and the value of $\nu_i$ on the finite sums of $H[i]$ has one of the five canonical forms. [For the reader's convenience: recall that $E:=\{2^n\mid n\in\omega\}$, and each $x\in\FS(E)$ (``finite sum of $E$'') has a unique support $\alpha(x)\subseteq E$ defined by $x=\sum \alpha(x).$]   

\medskip 

\noindent\underline{Diagonalisation.} At stage $k=\lranglet{i}{b}$, we handle critical node $\tau_i$. Our primary task is to recursively extract an increasing subsequence from the predetermined set
$$H[i]:=\{h_n\mid n\in\w\},$$
chosen to be sufficiently sparse so that the associated labels have relatively small harmonic weight. The choices will depend on which canonical form of $\nu_i$ holds on $\FS(H[i])$:
\smallskip

\begin{itemize}
	\item \textbf{Case (1).} Suppose
	$$\forall x,y\in\FS(H[i]).\;(\nu_i(x)=\nu_i(y)).$$
	Then there exists a constant $c\in\omega\cup\{\bot\}$ such that
	$$\FS(H[i])\subseteq\{x \mid \nu(\tau_i\fr x)=c\}\,.$$
	Since $\tau_i$ is critical, we know $\{x \mid \nu(\tau_i\fr x)=c\}\in\calH$. However, $H[i]$ is infinite by construction, and so $\FS(H[i])\notin\calH$, contradicting the inclusion above. Hence Case~(1) cannot occur.
	\medskip
	
	\item \textbf{Case (2).} Suppose 
	$$\forall x,y\in \FS(H[i]).\, \nu_i(x)=\nu_i(y)\iff \min \alpha(x)=\min \alpha(y)\,.$$
By the Canonical Hindman Theorem, we know 
	$$
	\max\alpha(h_n)<\min\alpha(h_{n+1})
	\qquad(n\in\omega). 
	$$
Hence,  deduce that
	$$
	\min\alpha(h_m)\neq \min\alpha(h_n), \qquad (\text{for any}\, m\neq n)
	$$
	and so by the canonical form,
	$$
	\nu_i(h_m)\neq \nu_i(h_n), \qquad (\text{for any}\, m\neq n).
	$$
	In sum, $\nu_i$ is injective on $H[i]$. 

	Returning to the task of recursively extracting a subsequence from $H[i]$, suppose we are at stage $k=\langle i,b\rangle$ of the diagonalisation. By Convention~\ref{conv:enum}, this means we have already chosen $$h_{n_0},\dots, h_{n_{b-1}}\in H[i]$$ in the previous stages. Since $\nu_i$ is injective on an infinite set $H[i]$, it has unbounded image in $\w$. In particular, examining Equation~\eqref{eq:partition-map}, this means we can always choose some $ n_b> n_{b-1}$ such that $\nu(\tau_i\fr h_{n_b})> 2^k$. 
	
	\smallskip
	
	Continuing this process allows us to construct three sequences
$$\{h_{n_b}\}_{b\in\w}\qquad \{D_{\lranglet{i}{b}}\}_{b\in\w}\qquad \{C_{\lranglet{i}{b}}\}_{b\in\w}$$
whereby:
\begin{itemize}[label=$\diamond$]
	\item $\{h_{n_b}\}_{b\in\w}$ is an increasing subsequence in $H[i]$, as constructed above;
	\smallskip
	
	\item 
	$D_{\lranglet{i}{b}}:=\bigg\{ \displaystyle\sum_{t\in F} h_{n_t}\Bmid F\in [\w]^{<\w}, F\neq\emptyset , \min F = b \bigg\}\,.$
	
	That is, $D_{\lranglet{i}{b}}$ collects all finite sums whose minimum summand is $h_{n_b}$.
	\smallskip
	
	\item $C_{\lranglet{i}{b}}:= \{ \nu(\tau_i\fr x) \mid x\in D_{\lranglet{i}{b}}\}$.
\end{itemize}

\medskip

\noindent In particular, they satisfy the following conditions:
\begin{enumerate}[label=(\alph*)]
	\item {\em (Well-labelled).} For every $b\in\w$, we have $\nu(\tau_i\fr h_{n_b})\neq \bot$; in particular, $C_{\lranglet{i}{b}}\subseteq \w$.
	
	[Why? This was implicit in our construction of $\{h_{n_b}\}_{b\in\w}$, but we elaborate. Since $\nu_i$ is injective, there exists at most one $x\in H[i]$ such that $\nu_i(x)=\lranglet{0}{0}$. Thus, deduce there exists at most one $x\in H[i]$ such that $\nu(\tau_i\fr x)=\bot$, which we can always avoid when picking $h_{n_b}$ since $H[i]$ is infinite.]
	
	\smallskip
	
	\item {\em ($\calH^*$-positivity).} Define
	$D[i]:=\bigcup_{b\in\w} D_{\lranglet{i}{b}}.$ Then
	$$D[i]\notin \calH\,.$$
	
	\noindent [Why? Every element of $\FS \left( \{h_{n_b} \mid b\in\w \}\right)$ has a unique minimum summand $h_{n_b}$ for some $b\in\w$. In particular, this means
	$$ \bigcup_{b\in\w} D_{\lranglet{i}{b}}= \FS \left( \{h_{n_b}\mid b\in\w \}\right).$$ 
	Since $\{h_{n_b}\}_{b\in\w}$ is strictly increasing, it defines an infinite subset of $\w$. Thus conclude that $D[i]=\FS \left( \{h_{n_b}\mid b\in\w \}\right) \notin\calH.$ ]
	
	\smallskip
	
	\item {\em ($\Sumn$-smallness).} At every stage $k=\lranglet{i}{b}$, 
	$$\sum_{c\in C_{\lranglet{i}{b}}}\frac{1}{c+1}< \frac{1}{2^k}\,.$$
	
	\noindent [Why? By construction, 
	$$\min \alpha(x) =  h_{n_b} \qquad\text{for all}\, x\in D_{\lranglet{i}{b}}\,.$$
Applying the canonical form of $\nu_i$, this implies $$C_{\lranglet{i}{b}}=\{\nu(\tau_i \fr h_{n_b})\}.$$
	Since $\nu(\tau_i\fr h_{n_b})>2^k$ by construction, the claim thus follows.]
\end{enumerate}
	
	\medskip 
	
\item \textbf{Case (3)}. Suppose $\nu_i$ has the canonical form:
	$$ \nu_i(x)=\nu_i(y)\iff \max \alpha(x)=\max \alpha(y)\,.$$
The argument is completely analogous to Case (2). Start by showing that 	$$
\max\alpha(h_m)\neq \max\alpha(h_n), \qquad (\text{for any}\, m\neq n)\,,
$$
and so $\nu_i$ is injective on $H[i]$. Then, recursively construct three sequences
	$$\{h_{n_b}\}_{b\in\w}\qquad \{D_{\lranglet{i}{b}}\}_{b\in\w}\qquad \{C_{\lranglet{i}{b}}\}_{b\in\w},$$
defined almost identically to Case (2) except that now
$$	D_{\lranglet{i}{b}}:=\bigg\{ \displaystyle\sum_{t\in F} h_{n_t}\Bmid F\in [\w]^{<\w}, F\neq\emptyset , \max F = b \bigg\}\,.$$
That is, $D_{\lranglet{i}{b}}$ collects all finite sums with maximum summand $h_{n_b}$. The sequences are well-labelled and $\Sumn$-small for the same reasons as in Case (2). For $\calH^*$-positivity, notice each finite sum has a unique greatest summand, and so
$ \bigcup_{b} D_{\lranglet{i}{b}}= \FS \left( \{h_{n_b}\mid b\in\w \}\right).$

\medskip
\item \textbf{Case (4)}. 	Suppose $\nu_i$ has the canonical form:
$$\nu_i(x)=\nu_i(y)\iff \min \alpha(x)=\min \alpha(y)\,\,\text{and}\,\,  \max \alpha(x)=\max \alpha(y)\,.$$
From this, we shall recursively extract a suitable subsequence from $H[i]$. Suppose we are at stage $k=\lranglet{i}{b}$ of the diagonalisation, and we have already chosen
$$h_{n_0}, \dots, h_{n_{b-1}}\in H[i]\,.$$
Since $\max \alpha(h_n)< \min \alpha (h_{n+1})$ for every $n\in\w$, deduce that 
$$\min \alpha(h_n+h_{n_r})\neq \min \alpha (h_n + h_{n_s})\quad\text{\and}\quad \min \alpha(h_n)\neq \min \alpha (h_n + h_{n_s})\,,$$
for every $n>n_{b-1}$ and $r<s\leq b-1$. Hence, appealing to its canonical form, deduce that the map $\nu_i$ is injective on
$$X_b:=\left\{\,h_n+h_{n_s} \mid n>n_{b-1},\, s\leq b-1\,\right\}\cup \left\{\,h_n\mid n>n_{b-1}\,\right\}\,.$$
Now fix 
\begin{equation} M:=(k+1)2^k.
\end{equation}
Since $\nu_i$ is injective on an infinite set $X_b$, at most finitely many elements $x\in X_b$ can satisfy $\nu_i(x)\leq \lranglet{1}{M}$. Equivalently, at most finitely many elements $x\in X_b$ have ``bad'' $\nu$-labels, i.e. either $\nu(\tau_i\fr x)=\bot$ or $\nu(\tau_i\fr x)\le M$. Now for each $n>n_{b-1}$, denote
$$P_n:=\{h_n\}\cup\{h_n+h_{n_s}\mid s\leq b-1\}.$$
Then $|P_{n}|=b+1$, and the sets $P_n$ are pairwise disjoint. Since only finitely many elements of $X_b$ have bad labels, pairwise disjointness means only finitely many packets $P_{n}$ meet this bad set. Hence, for some $n_b>n_{b-1}$, every element of $P_{n_b}$ has labels exceeding $M$, i.e.
\begin{equation}\label{eq:Hindman-Case4}
\nu(\tau_i\fr h_{n_b})> (k+1) 2^k
\quad\text{and}\quad
\nu(\tau_i\fr(h_{n_b}+h_{n_s}))> (k+1) 2^k
\end{equation}
for all $s<b$.
\smallskip

Continuing this process allows us to construct three sequences
$$\{h_{n_b}\}_{b\in\w}\qquad \{D_{\lranglet{i}{b}}\}_{b\in\w}\qquad \{C_{\lranglet{i}{b}}\}_{b\in\w}$$
whereby:
\begin{itemize}[label=$\diamond$]
	\item $\{h_{n_b}\}_{b\in\w}$ is an increasing subsequence in $H[i]$, as constructed above;
	\smallskip
	
	\item  $	D_{\lranglet{i}{b}}:=\bigg\{ \displaystyle\sum_{t\in F} h_{n_t}\Bmid F\in [\w]^{<\w}, F\neq\emptyset , \max F = b \bigg\}\,.$
	
That is, $D_{\lranglet{i}{b}}$ collects all finite sums whose maximum summand is $h_{n_b}$.
	\smallskip
	
	\item $C_{\lranglet{i}{b}}:= \{ \nu(\tau_i\fr x) \mid x\in D_{\lranglet{i}{b}}\}$.
\end{itemize}

\medskip
In particular, they satisfy the following conditions. 
\begin{enumerate}[label=(\alph*)]
	\item {\em (Well-labelled)}. For every $b\in\w$, we have $\nu(\tau_i\fr h_{n_b})\neq \bot$; in particular, $C_{\lranglet{i}{b}}\subseteq \w$.
	
		[Why? By construction of $\{h_{n_b}\}_{b\in\w}$, every element of the chosen packet $P_{n_b}$ has label exceeding $M=(k+1)2^k$; in particular, $\nu(\tau_i\fr h_{n_b})\in\w$ and cannot equal $\bot$.]
	\smallskip 
	\item {\em ($\calH^*$-positivity).} Define
	$D[i]:=\bigcup_{b\in\w} D_{\lranglet{i}{b}}.$ Then
	$$D[i]\notin \calH\,.$$
	[Why? Same as Case (3).]
	
	\smallskip 
	\item {\em ($\Sumn$-smallness).}  At every stage $k=\lranglet{i}{b}$, 
	$$\sum_{c\in C_{\lranglet{i}{b}}}\frac{1}{c+1}< \frac{1}{2^k}\,.$$
	
	[Why? Every element of $\FS(\{h_{n_b} \mid b\in\w\})$ has a unique minimum and maximum summand. Hence, represent $D_{\lranglet{i}{b}}$ as
	$$\qquad\qquad\displaystyle D_{\lranglet{i}{b}}=\{h_{n_b}\}\cup \displaystyle\bigcup_{s<b}\bigg\{ h_{n_s} + \sum_{r\in F} h_{n_r} + h_{n_b} \Bmid \, F\subseteq \{s+1, \dots, b-1\} \bigg\},$$ 
with the understanding that we set $F=\emptyset$ whenever $s+1=b$ . In particular, for fixed $s<b$, define 
	$$A_{s,b}:=\bigg\{ h_{n_s} + \sum_{r\in F} h_{n_r} + h_{n_b} \Bmid  F\subseteq \{s+1, \dots, b-1\}\bigg\}.$$
	 By construction, 
	$$\min \alpha (x) = \min \alpha (y) \,\text{and}\, \max \alpha(x)=\max \alpha(y)\qquad \text{for all}\, x,y\in A_{s,b}\,.$$
Hence, applying the canonical form of $\nu_i$, this implies 
$$\left\{\nu(\tau_i\fr x)\Bmid x\in A_{s,b}\right\} = \left\{\nu\left(\tau_i\fr (h_{n_s}+ h_{n_b})\right)\right\}\,\qquad\text{for every}\, s<b\,.$$
By construction, the sequence $\{h_{n_b}\}_{b\in\w}$ satisfies Equation~\eqref{eq:Hindman-Case4}. Hence, compute
\begin{align*}
\sum_{c\in C_{\lranglet{i}{b}}} \frac{1}{c+1} & = \frac{1}{\nu(\tau_i\fr h_{n_b}) +1} \;\; +\;\; \sum_{s<b}\frac{1}{\nu\left(\tau_i\fr (h_{n_s}+ h_{n_b})\right)+1} \\
& < \frac{1}{(k+1)2^{k}+1}\;\; +\;\; \sum_{s<b} \frac{1}{(k+1)2^{k}+1} \\
& < \frac{b+1}{(k+1)2^{k}} \leq \frac{1}{2^k}\,,
\end{align*}
where the final inequality uses the fact that $b\leq k=\lranglet{i}{b}$ for all $b\in\w$ (Convention~\ref{conv:enum}).]
\end{enumerate}

\smallskip 
\item \textbf{Case (5)}. Suppose $\nu_i$ has the canonical form:
$$\nu_i(x)=\nu_i(y)\iff x=y\,.$$
That is, $\nu_i$ is injective on $\FS(H[i])$. As before, we will use this to extract a suitable subsequence from $H[i]$. Suppose we are at stage $k=\lranglet{i}{b}$ of the diagonalisation, and we have already chosen
$$h_{n_0}, \dots, h_{n_{b-1}}\in H[i]\,.$$
Now fix
\begin{equation}
M:=2^{2k}.
\end{equation}
Since $\nu_i$ is injective on the infinite set $\FS(H[i])$, at most finitely many elements $x\in \FS(H[i])$ can satisfy $\nu_i(x)\leq \lranglet{1}{M}$. Hence, by similar reasoning as in Case (4), the set
$$X_b:=\big\{ x\in \FS(H[i])\bmid \nu(\tau_i\fr x)=\bot \,\,\text{or}\,\, \nu(\tau_i\fr x)\leq M \big\}\cup \{h_{n_0}\,, \dots,\, h_{n_{b-1}}\}$$
is finite. Hence, we may choose $n_b>n_{b-1}$ such that
\begin{equation}\label{eq:Hindman-Case5}
h_{n_{b}}>\max X_{b}\qquad\text{and}\qquad  \nu(\tau_i\fr h_{n_b})>M>2^{2k}\,.
\end{equation}

Continuing this process allows us to construct three sequences
$$\{h_{n_b}\}_{b\in\w}\qquad \{D_{\lranglet{i}{b}}\}_{b\in\w}\qquad \{C_{\lranglet{i}{b}}\}_{b\in\w}$$
whereby:
\begin{itemize}[label=$\diamond$]
	\item $\{h_{n_b}\}_{b\in\w}$ is an increasing subsequence in $H[i]$, as constructed above;
	\smallskip
	
	\item  $	D_{\lranglet{i}{b}}:=\bigg\{ \displaystyle\sum_{t\in F} h_{n_t}\Bmid F\in [\w]^{<\w}, F\neq\emptyset , \max F = b \bigg\}\,.$
	
	That is, $D_{\lranglet{i}{b}}$ collects all finite sums whose maximum summand is $h_{n_b}$.
	\smallskip
	
	\item $C_{\lranglet{i}{b}}:= \{ \nu(\tau_i\fr x) \mid x\in D_{\lranglet{i}{b}}\}$.
	
	\smallskip
\end{itemize}
That these sequences are well-labelled and satisfy $\calH^*$-positivity is completely analogous to the previous cases. To show $\Sumn$-smallness, we first record two observations:
\begin{itemize}[label=$\diamond$]
	\item For every $x\in \FS(\{h_{n_t} \mid t<b\})$, Equation~\eqref{eq:Hindman-Case5} yields
	$$h_{n_b} + x > h_{n_b}>\max X_b\,,$$
	and so
	$$	\nu(\tau_i\fr (h_{n_b}+x)) > M > 2^{2k}\,.$$
	In particular, every $c\in C_{\lranglet{i}{b}}$ satisfies the lower bound
		\begin{equation}\label{eq:Hindman-Case5a}
	c >M> 2^{2k}\,.
	\end{equation}
	\item $D_{\lranglet{i}{b}}$ is a finite set, and has size 
\begin{equation}\label{eq:Hindman-Case5b}
|D_{\lranglet{i}{b}}|=2^b 
\end{equation}
\end{itemize}
Hence, let us fix $k=\lranglet{i}{b}$, Combining Equations~\eqref{eq:Hindman-Case5a}-\eqref{eq:Hindman-Case5b}, compute:
\begin{align*}
\sum_{c\in C_{\lranglet{i}{b}}} \frac{1}{c+1} \leq (2^b)\cdot \frac{1}{2^{2k}+1} < \frac{2^b}{2^{2k}} \leq \frac{1}{2^k}\,,
\end{align*}
where, as in previous cases, the final inequality uses the fact that $b\leq k=\lranglet{i}{b}$ for all $b\in\w$.
\end{itemize}
{\bf \em This completes the case analysis, and the setup of our diagonalisation scheme.}

\medskip

\noindent \underline{Verification.} There are two things to check.
\begin{itemize}
	\item Collect all the successor labels into the set
	$$C:=\bigcup_{k\in\w} C_k\,.$$
	By $\Sumn$-smallness of the sequence $\{C_k\}_{k\in\w}$, compute that
	$$\sum_{c\in C} \frac{1}{c+1} \leq  \sum_{k\in\w}\, \sum_{c\in C_k} \frac{1}{c+1} <  \sum_{k\in\w} \frac{1}{2^k}<\infty\,, $$
	and so $C\in \Sumn$, as desired.
	\item Under each critical node $\tau_i$, we defined an $\calH^*$-positive set of successors
	$$D[i]\notin\calH\, , $$
	and every $x\in D[i]$ satisfies  $\nu(\tau_i\fr x)\in C$ by construction.
\end{itemize}

\medskip
\noindent \underline{Contradiction.} Since $C\in\Sumn$, its complement $A:=\w\setminus C$ belongs to the dual filter $(\Sumn)^*$. As $\Phi$ witnesses $\Sumn\glt\calH$, there exists an $\calH^*$-branching tree $T_A$ such that $\Phi[T_A]\subseteq A$. By Lemma~\ref{lem:separation-2}, the tree $T_A$ has a critical node $\sigma=\tau_i$. By definition, $\{n\mid \tau_i\fr n \in T_A\}\in \calH^*$. Since $D[i]$ is $\calH^*$-positive, there exists an $x\in D[i]$ such that $\tau_i\fr x\in T_A$. Now set $c:=\nu(\tau_i\fr x)$. By construction, $c\in C$. On the other hand, Lemma~\ref{lem:separation-1} yields some branch $p\in[T_A]$ with $\Phi(p)=c$. Thus $c=\Phi(p)\in \Phi[T_A]\subseteq A$, contradicting $c\in C=\omega\setminus A$. Hence, we have successfully shown $\Sumn\not\glt\calH$.
\end{proof}

\subsection{$\Sumn$ and $\Ram$ are incomparable}\label{sec:Ram} Similarly, we show that the canonical summable ideal and Ramsey ideal are incomparable in $\glt$, generalising a previous theorem by Filip\'{o}w and Kowitz \cite[Theorem 4.1]{FKK24}, who established the result in the setting of the classical Kat\v{e}tov order. 
\medskip

\begin{theorem}\label{thm:Sumn-Ram} $\Sumn$ and $\Ram$ are incomparable in the Gamified Kat\v{e}tov order. That is,
	$$\Sumn\not \glt \Ram\qquad\text{and}\qquad \Ram\not\glt \Sumn\,\,.$$
\end{theorem}
\begin{proof} Some conventions: we view $\Ram$ as an ideal on $\w$ via the fixed bijection $\code{-}\colon [\w]^2\simeq \w$ in Convention~\ref{conv:enum}. Unordered pairs will typically be denoted as $x\in [\w]^2$; whenever relevant, we may represent $x$ explicitly as  $x=\{s,t\}$, or by its enumeration $\code{x}\in\w$.
\medskip

By Corollary~\ref{cor:first-neg-H-R}, we already know that $\Ram\not\glt\Sumn$. It remains to show $\Sumn\not\glt \Ram$. The proof is completely analogous to Theorem~\ref{thm:Sumn-Hind}, modulo some minor adjustments. Suppose for contradiction there exists a partial continuous function $\Phi$ witnessing $\Sumn\glt \Ram$. Denote the labelling function as $\nu:=\nu^{\Ram}_\Phi$, and construct the canonical tree $T_\Phi$. As before, assume $\nu(\o)=\bot$ and enumerate all the critical nodes $\{\tau_i\}_{i\in\w}$ in the canonical tree $T_\Phi$. This sets up the following diagonalisation scheme.
\smallskip 

\noindent\underline{Goal}. Our main objective is to construct a forbidden label set 
$$C\in \Sumn,$$	
such that for every critical node $\tau_i$, we obtain a $\Ram^*$-positive set
$$D[i]\notin \Ram\,,$$	
constructed so that $\nu(\tau_i\fr \code{x})\in C$ for all $x\in D[i]$.
\smallskip

\noindent \underline{Partition}. For each critical node $\tau_i$, the labelling function $\nu$ induces a map\footnote{To avoid confusion: $\lranglet{-}{-}\colon\w\times \w \to\w$ denotes the fixed bijection on \textbf{ordered} pairs of $\w$, whereas $\code{-}\colon[\w]^2\to\w$ denotes the fixed bijection on distinct \textbf{unordered} pairs -- see Convention~\ref{conv:enum}.}
\begin{align}\label{eq:partition-map-Ram}
\nu_i\colon [\w]^2&\longrightarrow \w\\
x
& \longmapsto \begin{cases}
\lranglet{0}{0}\qquad\qquad\qquad\,\,\,\,\text{if}\,\nu(\tau_i\fr \code{x})=\bot\\
\lranglet{1}{\nu(\tau_i\fr \code{x})} \qquad\quad \text{if}\,\nu(\tau_i\fr \code{x})\in \w
\end{cases} \,\,.\nonumber
\end{align}
 Applying the Canonical Ramsey Theorem (Theorem ~\ref{thm:Ramsey}), fix an infinite set
\begin{equation}\label{eq:R-Ramsey}
T[i]\subseteq \w
\end{equation}
such that the value of $\nu_i$ on $\big[T[i]\big]^2$ has one of the four canonical forms. 

\smallskip 

\noindent\underline{Diagonalisation}. At stage $k=\lranglet{i}{b}$, we handle the critical node $\tau_i$. Our present task is to extract an increasing subsequence from the predetermined set 
$T[i]$ whose associated labels have relatively small harmonic weight. The choice of subsequence will depend on which canonical form of $\nu_i$ holds on $\Tit$:
\begin{itemize}
	\item \textbf{Case (1).} Suppose 
	$$\forall x,y\in\Tit. \,\, \nu_i(x)=\nu_i(y)\,.$$
	Then, there exists a constant $c\in\w\cup\{\bot\}$ such that 
	$$\Tit\subseteq \{x \mid \nu(\tau_i\fr \code{x})=c\}\,.$$
	Since $\tau_i$ is critical, this means $\{x\mid \nu(\tau_i\fr \cx)=c\}\in\Ram$. But $\Tit\notin\Ram$, contradiction. Hence Case (1) cannot occur.
	
	\smallskip
	
	\item \textbf{Case (2).} Suppose 
		$$\forall x,y\in\Tit. \,\, \nu_i(x)=\nu_i(y)\iff \min x = \min y\,.$$
	In which case, for any $t\in T[i]$, the restriction of $\nu_i$ to the set
	$$M_t:=\left\{\{s,t\}\mid s\in T[i] \,, s> t\right\}$$
	is constant. Hence, for any $t$, there exists a label $c_t\in\w$ such that
	$$\{c_t\}=\{\nu( \tau_i\fr \code{x}) \mid x\in M_t \}\,.$$
Moreover, whenever $t\neq t'$ in $T[i]$, we have
	$$\nu_i(x)\neq \nu_i(y)\qquad \text{for any}\, (x,y)\in M_t\times M_{t'}\,.$$
This means the labels $c_t$ are pairwise distinct, and so we can find a strictly increasing sequence $$\{t_{b}\}_{b\in\w}\subseteq T[i]$$ 
such that $c_{t_b}>2^{k}$ for every $k=\lranglet{i}{b}$.
This procedure allows us to construct three sequences 
$$\{t_b\}_{b\in\w}\, \qquad \,\{D_{\lranglet{i}{b}}\}_{b\in\w}\,\qquad \{C_{\lranglet{i}{b}}\}_{b\in\w}\,$$
whereby:
\begin{itemize}[label=$\diamond$]
	\item $\{t_b\}_{b\in\w}$ is a sequence of $T[i]$, constructed above;
	\item $D_{\lranglet{i}{b}}:=\left\{\{t_r,t_b\}\mid r> b\,\right\}$;
	\item $C_{\lranglet{i}{b}}:=\{\nu(\tau_i\fr \code{x})\mid x\in D_{\lranglet{i}{b}}\}$.
\end{itemize}
These satisfy the following conditions:
\begin{enumerate}[label=(\alph*)]
	\item {\em (Well-labelled).} For every $b\in\w$, we have $\bot\notin C_{\lranglet{i}{b}}$; in particular, $C_{\lranglet{i}{b}}\subseteq \w$.
	
	[Why? Since the assignment 
	\begin{align*}
	T[i]&\longrightarrow  \w\cup\{\bot\}\\
	t&\longmapsto c_t\,
	\end{align*}
	is injective, there exists at most one $t_\bot\in T[i]$ such that 
	$t_\bot\mapsto \bot$. Since $T[i]\setminus t_\bot$ is still an infinite set, we may assume without loss of generality that $c_t\neq\bot$ for all $t\in T[i]$.]

	\smallskip
	
	\item {\em ($\Ram^*$-positivity).} 	Define
	$D[i]:=\bigcup_{b\in\w} D_{\lranglet{i}{b}}.$ Then
	$$D[i]\notin \Ram\,.$$

	[Why? Define $H:=\{t_{b}\mid b\in\w\}$. Since any $x\in [H]^2$ may be represented as a pair $\{t_{r},t_{b}\}$ where $r> b$,
	$$D[i]=\displaystyle\bigcup_{b\in\w} D_{\lranglet{i}{b}}=[H]^2\,.$$
	Moreover, $H\subseteq\w$ is clearly infinite since $\{t_{b}\}_{b\in\w}$ is a strictly increasing sequence. Thus conclude $D[i]\notin \Ram$.]
	
	\smallskip
	
	\item {\em ($\Sumn$-smallness).} At every stage $k=\lranglet{i}{b}$, 
	$$\sum_{c\in C_{\lranglet{i}{b}}} \frac{1}{c+1}<\frac{1}{2^k}\,.$$
	
	[Why? By our above analysis, 
	$$c_{t_b}=\nu(\tau_i\fr\code{x})=\nu(\tau_i\fr\code{y})\qquad\text{for all} \,\,x,y\in D_{\lranglet{i}{b}}\,.$$
Since $c_{t_b}>2^{k}$ by construction, the claim thus follows.]
\end{enumerate}

\medskip 

\item \textbf{Case (3).} Suppose 
$$\forall x,y\in\Tit. \,\, \nu_i(x)=\nu_i(y)\iff \max x = \max y\,.$$
This case proceeds exactly as in Case (2) after making the obvious adjustments (swapping ``$r>b$'' with ``$r< b$'').
\medskip

\item \textbf{Case (4).} Suppose 
$$\forall x,y\in\Tit. \,\, \nu_i(x)=\nu_i(y)\iff   x =  y\,.$$
That is, $\nu_i$ is injective on $\Tit$. How might we leverage this to extract a suitable subsequence from $T[i]$? Suppose we are at stage $k=\lranglet{i}{b}$ of the diagonalisation, and we have already chosen
$$t_{0}, \dots, t_{{b-1}}\in T[i]\,.$$
Now fix
\begin{equation}
M:=k\cdot 2^k\,.
\end{equation}
Since $\nu_i$ is injective on the infinite set $\Tit$, at most finitely many elements $x\in \Tit$ can satisfy $\nu_i(x)\leq \lranglet{1}{M}$. Hence, the set
$$X_b:=\bigg\{ x\in \Tit \Bmid \nu(\tau_i\fr \cx)=\bot \,\,\text{or}\,\, \nu(\tau_i\fr \cx)\leq M \bigg\}$$
is finite. In particular, only finitely many elements of $T[i]$ occur in $X_b$, so we may choose
$$t_b\in T[i]\setminus \{t_{0}\,, \dots,\, t_{{b-1}}\}$$ such that
\begin{equation}\label{eq:Ramsey-Case4}
\nu(\tau_i\fr\code{\{t,t_b\}})>  k\cdot 2^k\,\qquad\text{for all}\,t\in T[i]\,.
\end{equation}
This procedure allows us to construct three sequences 
$$\{t_b\}_{b\in\w}\,\qquad \{C_{\lranglet{i}{b}}\}_{b\in\w}\, \qquad \,\{D_{\lranglet{i}{b}}\}_{b\in\w}\,$$
whereby:
\begin{itemize}[label=$\diamond$]
	\item $\{t_b\}_{b\in\w}$ is a sequence of elements in $T[i]$, constructed above;
	\item $D_{\lranglet{i}{b}}:=\left\{\{t_r,t_b\}\mid r< b\,\right\}$;
	\item $C_{\lranglet{i}{b}}:=\{\nu(\tau_i\fr \code{x})\mid x\in D_{\lranglet{i}{b}}\}$.
\end{itemize}
These satisfy the following conditions:
\begin{enumerate}[label=(\alph*)]
	\item {\em (Well-labelled).} For every $b\in\w$, we have $\bot\notin C_{\lranglet{i}{b}}$; in particular, $C_{\lranglet{i}{b}}\subseteq \w$.
	
	[Why? By construction, we define the sequence $\{t_b\}_{b\in\w}$ by choosing $t_b\in T[i]$ such that $(t,t_b)\notin X_b$ for all $b\in\w$. In particular, this means the sequence excludes any $x\in \Tit$ such that $\nu(\tau_i\fr[x])=\bot$.]

	\smallskip
	
	\item {\em ($\Ram^*$-positivity).} 	Define
	$D[i]:=\bigcup_{b\in\w} D_{\lranglet{i}{b}}.$ Then
	$$D[i]\notin \Ram\,.$$
	
	[Why? Define $H:=\{t_{b}\mid b\in\w\}$. Since any $x\in [H]^2$ may be represented as a pair $\{t_{r},t_{b}\}$ where $r<b$,
	$$D[i]=\displaystyle\bigcup_{b\in\w} D_{\lranglet{i}{b}}=[H]^2\,.$$
	Moreover, $H\subseteq\w$ is clearly infinite since we chose $t_b\in T[i]\setminus \{t_{0}\,, \dots,\, t_{{b-1}}\}$ at every stage $k=\lranglet{i}{b}$. Thus conclude $D[i]\notin \Ram$.]
	
	\smallskip
	
	\item {\em ($\Sumn$-smallness).} At every stage $k=\lranglet{i}{b}$, 
	$$\sum_{c\in C_{\lranglet{i}{b}}} \frac{1}{c+1}<\frac{1}{2^k}\,.$$
	
	[Why? Applying Equation~\eqref{eq:Ramsey-Case4}, compute 
	$$\sum_{c\in C_{\lranglet{i}{b}}} \frac{1}{c+1}=\sum_{r<b}\frac{1}{\nu(\tau_i\fr \code{\{t_r,t_b\}})+1}\leq\sum_{r<b}\frac{1}{k2^k +1} < \frac{b}{k {2^k}} \leq \frac{1}{2^k},$$
where again the final inequality uses Convention~\ref{conv:enum} to get $b\leq k=\lranglet{i}{b}$.]
\end{enumerate}
\end{itemize}
{\bf \em This completes the case analysis, and the setup of our diagonalisation scheme.}

\medskip

\noindent \underline{Verification \& Contradiction}. The rest of the proof proceeds exactly as in Theorem~\ref{thm:Sumn-Hind}. Start by defining
$$C:=\bigcup_{k\in\w} C_k\,.$$
One easily verifies that $C\in\Sumn$, and so $A=\omega\setminus C\in(\Sumn)^*$. By hypothesis, $\Phi$ witnesses $\Sumn\glt\Ram$, so there exists a $\Ram^*$-branching tree $T_A$ such that $\Phi[T_A]\subseteq A$. However, $T_A$ will always have some critical node $\tau_i$ with a successor from the $\Ram^*$-positive set $D[i]$, yielding a forbidden label in $C$, contradicting $\Phi[T_A]\subseteq A$. Thus, we conclude $\Sumn\not\glt\Ram$.
\end{proof}

\begin{conclusion} Non-linearity within the Gamified Kat\v{e}tov order is not restricted to the summable ideals; in fact, there exists several instances of non-linearity involving ``Ramsey-like'' ideals.
\end{conclusion}

\begin{discussion} By work of Kwela, it is also known that $\calH$ and $\Ram$ are incomparable within the classical Kat\v{e}tov order \cite[Theorems 5.2 and 6.1]{FKK24}. It is reasonable to conjecture that the same holds in the gamified setting, but we have not verified this.
\end{discussion}

 
 \section{Consequences for Computability Theory}\label{sec:topos}
 
 \subsection{Computability-Theoretic Background}
We now return to the setting of the Effective Topos $\Eff$, which originally motivated our study of the Gamified Kat\v{e}tov order.
To explain the place of our work within computability theory, we briefly review the major shifts in our understanding of LT topologies on $\Eff$ over the past several decades:
\begin{enumerate}
\item \textbf{First stage (1980s--2000s).} In his original paper \cite{HylandEffective}, Hyland showed that the Turing degrees embed effectively into the $\clt$-order on LT topologies in $\Eff$. This suggested viewing the LT topologies as a generalisation of Turing oracles (i.e. single-valued oracles), but the nature of this generalisation remained unclear for many years. A few examples of non-Turing oracles were subsequently discovered (most notably, the double negation topology, Pitts' counterexample \cite{PittsPhD}, and Lifshitz realisability \cite{vO91}), but these were generally regarded as exceptional cases rather than manifestations of a broader theory. 


\item \textbf{Second stage (2010s).}
Lee-van Oosten \cite{LvO13} provided concrete presentations of all LT topologies on ${\rm Eff}$.
They focused in particular on the {\em basic} LT topologies and found several new examples that do not arise from single-valued oracles. However, the combinatorial mechanism separating the basic topologies (or indeed LT topologies in general) was not yet identified.

\item \textbf{Third stage (2020s).}
Kihara \cite{Kih23}, building on Lee-van Oosten's work, linked the $\clt$-order to {\em Weihrauch reducibility}, a notion extensively studied in computable analysis \cite{BGP21}. This marked a shift in perspective on what should be regarded as the ``paradigmatic'' LT topology: in computable analysis, multi-valued oracles (Weihrauch oracles) are the primary objects of study, with single-valued oracles (Turing oracles) appearing as a special case.

\item \textbf{Fourth stage (2020s).} In fact, Kihara \cite{Kih23} linked the LT-topologies to the {\em extended} Weihrauch degrees introduced by Bauer \cite{Bau22}. The embedded copy of the Weihrauch degrees within this new hierarchy is called the {\em modest} degrees, with the remaining {\em non-modest degrees} being viewed as the exceptional case. Nevertheless, Kihara pointed out that each basic topology (in the sense of Lee-van Oosten) yields a non-modest degree; several natural examples were analysed in the cited paper, with the task of constructing further examples left as a future challenge \cite[Question 1]{Kih23}.
\item \textbf{Fifth stage (Present work).} Kihara and Ng \cite{KiNg26} linked the basic topologies to the combinatorial complexity of filters (dually, ideals) on $\w$ through the Gamified Kat\v{e}tov order. In particular, our results highlight how these topologies encode rich interactions between computable and combinatorial complexity. This discovery repositions the basic topologies, and thus the non-modest degrees, from isolated curiosities to central objects of study.\footnote{{
		\em Side-note.} In a recent paper \cite{ASM}, Abou Samra and Madore argue that the $\clt$-order arises naturally in computability theory by considering two forms of multivaluedness (non-determinism). The Weihrauch (modest) part captures only what they call {\em demonic} non-determinism, while the non-modest part incorporates {\em angelic} non-determinism.
	In their language, our work on the Computable Gamified Kat\v{e}tov Order thus provides a computability-theoretic analysis of pure angelic non-determinism.
}
\end{enumerate}
This fifth stage opens up a new research area previously overlooked in computability theory. Informally, filters and ideals correspond to abstract notions of majority and minority. Consequently, the $\clt$-order on basic topologies on $\Eff$ -- equivalently, the computable Gamified Kat\v{e}tov order -- measures the relative strength of what we call {\em computability by majority} notions (for details, see \cite[\S 7.2]{KiNg26}). This section extends the set-theoretic results established in the previous sections to illuminate several important structural aspects of this relatively uncharted territory.


\subsection{Non-linearity in the Computable Setting} By Theorem~\ref{thm:main-thm}, the $\clt$-order on filters (dually, ideals) is equivalent to the $\glt$-order where the witness map is now required to be computable. Leveraging this characterisation, many of the non-linearity results from before transfer immediately to the $\clt$-order.

\begin{theorem}\label{thm:LT-Pw} The extended Weihrauch lattice of non-modest degrees embeds a copy of $\Pw/\Fin$.
\end{theorem}
\begin{proof} The positive direction of Theorem~\ref{thm:Pw-emb} remains valid, since $\id\colon\w\to\w$ witnesses $\Sump\glt\Sumq$ for $P\subseteq^* Q$ and $\id$ is computable. The negative direction also transfers, since $\glt$ is coarser than $\clt$. 
And so we conclude
$$P\subseteq^*Q\iff \Sump\clt \Sumq\,.$$	
\end{proof}

\begin{theorem} The following pairs are incomparable in the $\clt$-order:
	\begin{enumerate}[label=(\roman*)]
		\item $\calH$ and $\Sumn$.
		\item $\Ram$ and $\Sumn$.
	\end{enumerate}
\end{theorem}
\begin{proof} Immediate from  Theorems~\ref{thm:Sumn-Hind} and \ref{thm:Sumn-Ram}, and the fact that $\glt$ is coarser than $\clt$.
\end{proof}

\begin{discussion} The $\clt$-order was already known to be non-linear on upper sets; Lee-van Oosten \cite[Proposition 5.11]{LvO13} gives one such instance, essentially by examining intersection patterns of subset families.\footnote{Technically, Lee-van Oosten show the result for the $\clt$-order on subset families, but the translation to upper sets is immediate since upward closure preserves the $\clt$-complexity of subset  families.} 
Our results substantially develop this picture -- for instance, by Theorem~\ref{thm:LT-Pw}, there exists an antichain of size $\mathfrak{c}$ within the $\clt$-order on upper sets (in fact, filters).
\end{discussion}

However, not every consequence of Section~\ref{sec:Pw} transfers automatically.  Unlike Theorem~\ref{thm:Pw-emb-Denz}, we cannot immediately deduce from 
Theorem~\ref{thm:LT-Pw} the existence of an embedding of $\Pw/\Fin$ between $\Fin$ and $\Denz$ in the $\clt$-order. In the non-computable setting, we know that $\calI\glt\Denz$ for any summable ideal $\calI$. This is no longer true in the computable setting, as below:

\begin{corollary}\label{cor:no-Denz} Let $\calU\subseteq\Pw$ be an upper $\Delta^1_1$-set. Then, there is no function $f\colon\w\to\w$ such that
	$$\calI^*\clt f\times \calU\,.$$
	for all summable ideals $\calI$.
\end{corollary}
\begin{proof} This follows from combining several previous results in \cite{KiNg26}. Fix $f$. By relativising \cite[Theorem F]{KiNg26}, we get the following statement: for any $g\colon\w\to\w$, $g\clt f\times\calU$ iff $g$ is hyperarithmetical relative to $f$. In particular, if $g$ is not hyperarithmetical relative to $f$, then $g\not\clt f\times\calU$. However, by the Cofinality Theorem \cite[Theorem E]{KiNg26}, there exists a summable ideal $\calI$ such that $g\clt \calI^*$. Hence, this implies $\calI^*\not\clt f\times\calU$. 
\end{proof}

\begin{discussion} Notice this does not contradict \cite[Theorem 6.13]{KiNg26}, which can be translated as:
	$$\calI^*\glt \calU \iff \calI^*\clt f\times \calU \qquad\text{for some\,$f\colon\w\to\w$\,,}$$ 
for any ideal $\calI$ and upper set $\calU$. In this case, the Turing oracle $f\colon\w\to\w$ may vary depending on the choice of ideal $\calI$. Corollary~\ref{cor:no-Denz} clarifies this by saying there is no {\em fixed} oracle $f$ that works uniformly for all summable ideals.
\end{discussion}

We thus close the paper with an open question:

\begin{problem}\label{prob:Pw-LT} Does there exist an embedding of $\Pw/\Fin$ into the $\clt$-order whose image lies between $\Fin$ and $\Denz$?
\end{problem}

    \bibliography{Eff}

\end{document}